\documentclass[a4paper, 11pt, twoside]{article}
\usepackage[utf8]{inputenc}
\usepackage[T1]{fontenc}
\usepackage{csquotes}
\usepackage{hyperref}
\usepackage{amsmath, amsfonts, amsthm, amssymb, amscd, mathtools, mathrsfs} 
\usepackage{graphicx, float, subcaption}  
\usepackage[dvipsnames]{xcolor} 
\usepackage{lipsum} 
\usepackage{enumerate} 
\usepackage{array} 
\usepackage[Lenny]{fncychap} 
\usepackage{bbold} 
\usepackage{indentfirst} 
\usepackage{wasysym}
\usepackage{bm} 
\usepackage{multirow, multicol, tabularx}
 \usepackage{pdfpages}
 \usepackage{ulem}
 \usepackage{upgreek}
\usepackage[]{mdframed}
\usepackage{cancel}
\usepackage{mathdots}

\usepackage{geometry}
\newgeometry{top=0.9in, bottom=0.9in, outer=0.8in,inner=0.8in}

\usepackage[skip=10pt, indent=15pt]{parskip}
\usepackage{hyperref}
\hypersetup{
	colorlinks=true,
	linkcolor=black,
	urlcolor=green,
	citecolor=blue,
	linktoc=all
}
 
\usepackage{fancyhdr}

 \newtheorem{theo}{\textbf{Theorem}\ }
[section]
\newtheorem{lemma}[theo]{\textbf{Lemma}\ }
\newtheorem{coro}[theo]{Corollary\ }

\newtheorem{prop}[theo]{\textbf{Proposition}\ }

\newtheorem{remark}[theo]{\textbf{Remark}\ }

\begin{document}

\title{A local limit theorem for lattice oscillating random walks }
\maketitle
\begin{center}
M. Peign\'e \footnote{
			Institut Denis Poisson UMR 7013,  Universit\'e de Tours, Universit\'e d'Orl\'eans, CNRS  France.  email: peigne@univ-tours.fr},
		{ C. Pham} \footnote{
			CERADE ESAIP, 18 rue du 8 mai 1945 - CS 80022 - 49180 St-Barth\'elemy d'Anjou; email: dpham@esaip.org}$^,$\footnote{
			LAREMA UMR CNRS 6093, Angers, France. } 
			$\&$
		{ T. D. Vo} \footnote{Department of Mathematics, FPT University, District 9, Ho Chi Minh City, Vietnam. email: duyvt15@fe.edu.vn} 
\end{center}

     \vspace{1cm}
     
\centerline{\bf  \small Abstract   \footnote{The three authors are supported by ANR-23-CE40-000}}

 In this paper, we obtain   a local limit theorem for the Kemperman's model of  oscillating random walk on $\mathbb{Z}$; it extends the existing results for classical  random walks on $\mathbb Z$ or reflected random walks on $\mathbb N_0$. The key technical point   is to control the long-term behavior of the embedding subprocess that characterizes the oscillations of the original random walk between $\mathbb Z^-$ and $\mathbb Z^+$ in both recurrent and transient cases. Then by combining   an extension of \cite[Theorem 1.4]{gouezel}  for the convergence of aperiodic sequence of renewal   operators acting on a suitable   functional Banach space  and the   decomposition of the trajectories of the random walk, we obtain the exact asymptotic for the return probability under some mild assumptions on the increment moments.

  \vspace{0.5cm} 60B12, 60B10, 82B41, 
  
  \noindent Keywords and phrases:   local limit theorem, oscillating random walk, Doob's transform, aperiodic sequence of operators



	\section{Introduction}

 It is a general fact that  the heat kernel  on a Riemannian manifold $M$ is a smooth function on $(0, +\infty)\times M\times M, (t, x, y) \mapsto p(t, x, y)$
whose  shape   depends on the properties of $M$.  In  \cite{saloff}, L. Saloff-Coste pays particular attention to the long-time, large-scale behavior of the heat kernel and its relation to the global geometry of $M$. To answer to this general question, L. Saloff-Coste   needs tools to obtain sharp two-sided estimates for the heat kernel in terms of the time variable $t > 0$ and basic geometric quantities depending on $x, y \in M$.  He studies in particular  the case of manifolds made of ``nice pieces" and propose a strategy of ``reconstruction"  to describe the behavior of the heat kernel on the whole manifold; this strategy can be extended  to manifolds with finitely many ends.

Similar questions are relevant in the case of discrete structures $\mathbb X= \mathbb X_0\cup \mathbb X_1\cup \ldots \cup\mathbb X_\ell$   with finitely many  ends $\mathbb X_1, \ldots, \mathbb X_\ell,$ where $\mathbb X_0$ is finite and  $\ell \geq 2$. We  consider an irreducible Markov chain  $(X_n)_{n \geq 0}$ on $\mathbb X$    with transition probability $P=(p(x, y))_{x, y \in \mathbb X}$. The aim   here is to describe the asymptotic behavior as $n \to +\infty$ of  the   probabilities $p^{(n)}(x, y)$, assuming that we have precise estimates of the quantities $\mathbb P[X_n=y\mid X_0=x, X_1\ldots, X_n \in \mathbb X_i]$ for any $1\leq i\leq \ell$.

 The simplest example of a such   structure is the ``oscillating random walk"  on $\mathbb Z$ introduced by J. H. B.  Kemperman  in 1974  \cite{kemperman} and defined by
  \begin{equation} \label{qjvhbe}
  X_0  =x \quad \text{and} \quad
X _{n+1}=\left\{
\begin{array}{cll}
X_n+  \xi_{n+1}  &if \quad X_n \leq -1,  
\\
 X_n+  \eta_{n+1} &if \quad X_n = 0,  
\\
X_n + \xi'_{n+1}  &if \quad X_n \geq 1, 
\end{array}
\right.   
\quad \text{for} \quad n \geq 0,
 \end{equation} 
 where $x$ is a fixed point in $\mathbb Z$ and $(\xi_i)_{i \geq 1}, (\eta_i)_{i \geq 1}$ and 
 $(\xi'_i)_{i \geq 1}$ are independent iid sequences of $\mathbb Z$-valued random variables with  respective distributions $\mu, \mu_0$ and $\mu'$.  This Markov chain models discrete-time diffusions  in one-dimensional space with three different media $\mathbb Z^{*-}, \{0\}$ and $\mathbb{Z^{*+}}$. Whenever the process $\mathcal{X}=(X_n)_{n \geq 0}$ lies on the negative half-line $\mathbb Z^{*-}$, its trajectory is directed by the left distribution $\mu$ until it quits $\mathbb Z^{*-}$;  once it  visits $\{0\}$ or enters the positive half-line $\mathbb Z^{*+}$, its trajectory is directed respectively by $\mu_0$ and the right distribution $\mu'$ until quitting   these sets...  and so on. 
  When $\mu=\mu_0=\mu'$, the process $(X_n)_{n \geq 0}$ is a classical random walk on $\mathbb Z$  whose behavior is well known.

J. H. B.   Kemperman obtained  an explicit  general necessary and sufficient condition for the recurrence of the oscillating random walk. B. A. Rogozin and S. G. Foss obtained in \cite{rogozinfoss} sufficient 
 conditions for recurrence   if the distributions $\mu$ and $\mu'$ belong
to the region of attraction of stable laws.
Recently, J. Br\'emont proposed in \cite{bremont} a straight line and self-contained exposition of these results  and T. D. Vo \cite{vo1}  obtained a sufficient condition in terms of moment of $\xi_1$ and $\xi'_1$.   B. A. Rogozin and S. G. Foss also  gave examples
illustrating that the oscillating random walk can be  transient  even in the case when $\mathbb E[\vert \xi_1\vert]$ and  $\mathbb E[\vert \xi'_1\vert]  $  are both finite and $\mathbb E [ \xi_1 ]= \mathbb E[ \xi'_1 ] =0$,  in contrast to  the classical  random walk. To recover the recurrence property, T. D. Vo   simply reinforces the moment hypotheses and assume that $\mathbb E[\vert \xi_1\vert^{1+p}]+\mathbb E[\vert \xi'_1\vert^{1+q}]<+\infty$ with $p+q\geq 1$.  This stronger hypothesis is not necessary of course, as demonstrated by the classical random walk; nevertheless, as soon as $p+q<1$, there exist  examples of transient ``doubly centered" oscillating random walks  (see for instance \cite{peignewoess}).

Note that  among the oscillating random walks,   we find the unfolded version on $\mathbb Z$ of the reflected walk on $\mathbb N_0 $ which has been the subject of intense studies for a few decades \cite{peigneessifi} and is a basic and important model of stochastic dynamical systems \cite{peignewoess}. Recently, G. Alsmeyer, S. Brofferio and D. Buraczewski \cite{alsmeyerbrofferioburaczewski} studied one dimensional iterated functions systems generated by asymptotically linear maps with possibly different slopes  at  $-\infty$ and $+\infty$; these systems are closed to classical affine iterations, but the  analysis differs according to the sign and the size of the slopes at infinity. Their results present analogies with the present article with a competition between the two slopes at infinity.

 Let us recall briefly some known limit theorems for classical random walks $(S_n)_{n\geq 0}$  on $\mathbb Z$, reflected random walks $(R_n)_{n \geq 0}$ on $\mathbb N_0$ and oscillating random walks on $\mathbb Z$. First of all, in the spirit of the Donsker's invariance principle for classical random walks, functional limit theorems  have been obtained  for the reflected and oscillating   random walks. Under quite general assumptions, the process $(R_n)_{n \geq 0}$, properly rescaled, converges towards  the reflected Brownian motion on $\mathbb R^+$  \cite{ngopeigne} while the  oscillating random walk  $(X_n)_{n \geq 0}$ converges to a skew Brownian motion whose parameter can be expressed in terms of renewal functions associated with the random walks of step distributions $\mu$ and $\mu'$ \cite{peignevo}.
Concerning the local limit theorem for   (irreducible and aperiodic)   classical  or reflected random walks, the spatial inhomogeneity introduced by the   reflection  at $0$ in the second case means that the asymptotic behavior of the return probabilities is markedly different from   one case to  the other.

$\bullet$ {\it Classical random walks $(S_n)_{n\geq 0}$  on $\mathbb Z$.} 
 If $\mathbb E[\vert \xi_1\vert^2]<+\infty$ and $\mathbb E[\xi_1]=0$  then   there exists a strictly positive constant $c$ which does not depend on $x$ and $y$ such that $\displaystyle  \mathbb P_x[S_n=y]\sim c/\sqrt{n}$ as $n \to +\infty$  \cite{spitzer}. 
When  $\xi_1$ is not centered, assuming  that $\mathbb E[e^{t\xi_1}]<+\infty$ for $t$ sufficiently large, there exists a strictly positive constant $c= c(x-y)$ such that   $\displaystyle  \mathbb P_x[S_n=y]\sim c\rho^n/\sqrt{n}$ as $n \to +\infty$, where $\rho$ is the minimum of the Laplace transform $t \mapsto \mathbb E[e^{t\xi_1}]$. These results are based on Fourier analysis and a classical Doob's transform to bring   back to the centered case.

$\bullet$ {\it Reflected random walks $  (R_n)_{n\geq 0}$  on $\mathbb N_0$.}   The process $(R_n)_{n\geq 0}$
  is defined recursively by
$R_0= x \in \mathbb N_0$ and $R_{n+1}=\vert R_n-\xi_{n+1}\vert $ for any $n \geq 0$.  The local limit theorem for  the process  $(R_n)_{n\geq 0}$ has been studied  in  \cite{peigneessifi} under the hypothesis $\mathbb E[e^{t\xi_1}]<+\infty$ for $t$ sufficiently large.  If  $\mathbb E[\xi_1]=0$  then   there exists a strictly positive constant $c$ depending on $y$   such that $\displaystyle  \mathbb P_x[S_n=y]\sim c \rho^n/\sqrt{n}$ as $n \to +\infty$, where $\rho$ is the minimum of the Laplace transform $t \mapsto \mathbb E[e^{t\xi_1}]$. 
When  $\xi_1$ is not centered,  there exists $c= c(x,y)>0$  such that   $\displaystyle  \mathbb P_x[S_n=y]\sim c\rho^n/n^{3/2}$ as $n \to +\infty$.  Notice that in the centered case, the constant $c$ does not depend on the starting point $x$  while in the non centered case it depends on both $x$ and $y$, and not only on the difference $x-y$ as for classical random walks. These  results are based on Darboux's method, via the study of the singularity of certain functions of the complex variable, with operator values.
 
In the present paper, we focus on the asymptotic behavior of the quantity $\mathbb P_x[X_n=y]$  for oscillating random walks on $\mathbb Z$ with   three media   $\mathbb Z^{*-}, \{0\}$ and $\mathbb Z^{*+}$. Somewhat surprisingly, it is when the two classical  random walks have a drift of the same sign that there are quite complex interactions between them.  Even the modest perturbation of the transitions from a finite set of sites (for instance transitions from the origin) would reveal other possible interactions, making the analysis   more complex.  Our  approach is based on the analysis of the transition probability of the ``switching  subprocess" of the oscillating random walk, which controls the transitions from one medium  $\mathbb Z^{*-}, \{0\}$ or $\mathbb Z^{*+}$  to another one. This strategy  is quite flexible  and can be adapted to other processes, for instance  ``perturbed random walks with a membrane"  \cite{pilipenkopryhodko} which is a quite active area for one decade.  The same strategy is also valid in  the case of oscillating random walks on $\mathbb R$; when one of the distributions $\mu, \mu'$ or $\mu_0$ is spread out, the transition operator of  the switching subprocess  falls within the scope of quasi-compact operators.

\section{Description of the model and hypotheses} \label{statement}

Let $(\xi_n)_{n \geq 1}, (\eta_n)_{n \geq 1}  $ and $(\xi'_n)_{n \geq 1}$ be    independent sequences of  iid  $\mathbb Z$-valued  random variables, defined on a probability space $(\Omega, \mathcal F, \mathbb{P})$ and with respective distributions $\mu, \mu_0$ and $\mu'$. 
The {\it oscillating random walk}   $\mathcal{X} =(X  _n)_{n\geq 0}$ on $\mathbb Z$  with jump distributions $ \mu, \mu_0 $ and $\mu'$ is defined recursively by \eqref{qjvhbe}. 
To emphasize the dependence in $\mu, \mu_0$ and $\mu'$   of this  process, we write $\mathcal X (\mu, \mu_0,   \mu')$. 
 
 We set $D:= \sup\{x \geq 1\mid \mu(x)>0\}, D':= \inf \{x\leq  -1\mid \mu'(x)>0\}, D_0^+:= \sup\{x \geq 1\mid \mu_0(x)>0\}$, $D_0^-:= \inf \{x\leq -1\mid \mu_0(x)>0\}$   and introduce the following assumptions:

  $\bf O1- $ {\it $(\xi_n)_{n \geq 1}, (\eta_n)_{n \geq 1} $ and $(\xi'_n)_{n \geq 1}$ are independent sequences of  iid  $\mathbb{Z}$-valued random variables}.

 $\bf O2-$ {\it   Both distributions $\mu $ and $\mu'$  are strongly aperiodic on $\mathbb Z$. }

  ${\bf O3-}\  DD'\leq -2.$

Hypothesis {\bf O3} implies in particular  that the supports of $\mu$ and $\mu'$ are neither included in $\mathbb Z^-$ nor in $\mathbb Z^+$. This last condition is weaker since it is equivalent to $DD'\leq -1$; it  is also classical  (and as far as we know, it is indispensable) in the theory of random walks in order to obtain local limit theorems in the non centered  case.   In subsection \ref{toy}, we explore briefly the case $DD'=-1$. It is simpler   since the successive excursions within the different media are   independent in this case. The hypothesis ${\bf O3}$  breaks this independence and ensures that  the oscillating random walk  can  jump, at least in one direction,   from one half line to the other without visiting $0$.

${\bf O4}-  \mu_0(\mathbb Z^{*-})\mu_0 (\mathbb Z^{*+})>0 $ {\it (in particular $\mu_0(0) <1$}).

By \cite{vo1},  under   hypotheses {\bf O1}, {\bf O2} and {\bf O3}, the process $\mathcal X$ is irreducible on $\mathbb Z$.
We   explore here the asymptotic behavior  as $n \to +\infty$ of the quantity $\mathbb{P}_x[X_n=y]$,  under a varied set of conditions of moment  on   $\mu$ and $\mu'$. Up to exchanging the  roles  of $\mu$ and $\mu'$, we have seven cases   to explore, according to the fact that their means  are negative {\bf (N)}, positive  {\bf (P)}  or zero  {\bf (Z)}. 

 In  the  cases $({\bf P,N})$, $({\bf Z,Z})$ and $({\bf P, Z})$,  the oscillating random walk is recurrent,  while in the other cases, it is transient. We introduce some moment hypotheses.

For any $\mathbb Z$-valued random variable $\xi$, we set ${\xi}^- := \max(0, -\xi)$ and ${\xi}^+ := \max(0, \xi)$.
 
 \noindent {\bf Moment conditions}
 
 \vspace {2mm}
 
\noindent   \underline{{\it Case} $({\bf P,N})$}  {\it where} $\mathbb E[\xi_1] >0 $ {\it and}  $ \mathbb E[\xi'_1]<0, $

{\it with the following moment assumptions :  } $\mathbb E[\vert \xi_1\vert],  \mathbb E[\vert \eta_1\vert], \mathbb E[\vert \xi'_1\vert]<+\infty$.

\vspace {2mm}

\noindent \underline{{\it Case} $({\bf Z,Z})$}   {\it where} 
  $\mathbb E[\xi_1] =  \mathbb E[\xi'_1]=0, $
  
 {\it with the following moment assumptions :  there exists $\delta_0>0$ such that}
$$
\mathbb E[\vert \eta_1\vert ^{1+\delta_0}], \mathbb E[\xi_1^2], \mathbb E[{({\xi^+_1})}^{3+\delta_0}],   \mathbb E[{\xi'_1}^2], \mathbb E[{({\xi'_1}^-)}^{3+\delta_0}] < +\infty.
$$

\vspace {2mm}

\noindent  \underline{{\it Case} $({\bf P,Z})$}    {\it where}  $\mathbb E[\xi_1] >0 $ {\it and}  $ \mathbb E[\xi'_1]=0, $

 {\it with the following moment assumptions :  there exists $\delta_0>0$ such that} 
$$
 \mathbb E[\vert \eta_1\vert ^{1+\delta_0}], \mathbb E[\xi_1^2],  \mathbb E[{\xi'_1}^2],   \mathbb E[{(\xi_1^+)}^{2+\delta_0}],  \mathbb E[{({\xi'_1}^-)}^{3+\delta_0}] < +\infty.
$$
 
\noindent \underline{{\it Case} $\bf (Z, P)$}  {\it where} $\mathbb E[\xi_1] =0 $ {\it and}  $ \mathbb E[\xi_1']>0, $

  {\it with the following moment assumptions :  there exist  $\delta_0>0$   such that} 

 $\bullet \    \mathbb E[\vert \eta_1\vert^{1+\delta_0}], \mathbb E[\xi_1^2], \mathbb E[(\xi^+_1)^{3+\delta_0}]<+\infty, $

 $\bullet\  
  \mathbb E[({\xi'_1})^2], \mathbb E[ ({\xi'_1}^-)^{5/2+\delta_0}] <+\infty.$

 For the last two cases below, we assume that the random variables  $\xi_1$ and $\eta_1$ have the same distribution; we   explain why in section \ref{sectiontransientNPPP}. Hence only moment hypotheses  of $\xi_1$ and $\xi'_1$ are needed.

\noindent \underline{{\it Case} $\bf (N, P)$}  {\it where}  $\mathbb E[\xi_1] <0 $  {\it and}   $ \mathbb E[\xi'_1]>0 $

{\it with the following moment assumptions : there exist  $\delta_0, \lambda>0$ and $\lambda'<0$  such that}

 $\bullet \  \displaystyle   \mathbb E[ e^{(\lambda+\delta_0) \xi_1}],  \mathbb E[ e^{ (\lambda' -\delta_0)\xi'_1}]<+\infty \quad and   \quad \mathbb E[\xi_1 e^{\lambda \xi_1}] = \mathbb E[\xi'_1 e^{\lambda' \xi'_1}]=0, $

\noindent \underline{{\it Case} $\bf (P, P)$}  {\it where}   $\mathbb E[\xi_1] >0 $ {\it and}  $ \mathbb E[\xi'_1]>0 $

 {\it with the following moment assumptions :   there exist  $\delta_0>0$ and $ \lambda, \lambda' <0$ such that
}

$\bullet \  \displaystyle \mathbb E[ e^{  (\lambda -\delta_0)\xi_1}], \mathbb E[ e^{   (\lambda'-\delta_0) \xi'_1}]<+\infty \quad and   \quad \mathbb E[\xi_1 e^{\lambda \xi_1}] = \mathbb E[\xi'_1 e^{\lambda' \xi'_1}]=0. $

By \cite{vo1},  in the  case  $({\bf P, N)}$,  the   oscillating random walk $\mathcal X$ is   positive recurrent  and  the classical theory of denumerable Markov chains yields 
$$ \mathbb P_x[X_n=y] \longrightarrow \nu(y)  \quad \text{as} \quad n \to +\infty,$$  
where $ \nu $ is the unique invariant probability measure of $\mathcal{X}$.  In the cases $({\bf Z, Z)}$ and $({\bf P, Z)}$, the   oscillating random walk $\mathcal X$ is null recurrent     and its study requires more attention.    As demonstrated in \cite{vo1},  in these cases, the embedding  switching subprocess $\mathcal{X_{\bf C}}$ (to be introduced in section \ref{subprocess})  is positive recurrent on its unique essential class and its transition probability operator acts as a compact operator on a  particular functional space $\mathcal{B}$. This property is of utmost importance for applying the Gou\"ezel's renewal theorem (see section \ref{sectionrecurrent} for the details), which provides a precise asymptotic behavior of the  probabilities $\mathbb{P}_x[X_n = y]$ for any $x, y \in \mathbb{Z}$.  
Notice that
in these three cases, the spectral radius $\rho_{\mathcal X}:= \lim_{n \to +\infty} \mathbb P_x[X_n=y]^{1/n}$ equals 1. 

 In the last three cases $\bf(Z, P)$, $\bf (N, P)$ and $\bf ( P, P)$, the oscillating random walk $\mathcal X$ is transient.  In the case $\bf(Z, P)$, although  $\mu'$ is non  centered, only sufficiently large polynomial moments are required  for $\mu'$  to obtain a suitable overestimation of the return probabilities (see section \ref{sectiontransientZP}).
 In  the remaining  cases $\bf (N, P)$ and $\bf ( P, P)$, in order to get   a precise asymptotic  of $\mathbb P_x[X_n=y]$, exponential moments  for  $\mu$ and $\mu'$  are required. We then   assume that  $\mathbb E[e^{t\xi_1}]$ and $\mathbb E[e^{t\xi'_1}]$ are finite for $\vert t\vert$ large enough and we   consider  the respective Laplace transforms $L$ and $L'$  of $\mu$ and $\mu'$ defined formally by
$$
  L(t):= \mathbb E[e^{t\xi_1}] \quad {\rm and}  \quad  L'(t):= \mathbb E[e^{t\xi'_1}].
 $$  These Laplace transforms are  well-defined and analytic on an open neighborhood  $\mathcal O$ of $0$; furthermore, under hypothesis {\bf O3}, these two functions are unbounded on each side of $0$ and   they thus reach their  minima  on this open set. We denote by $\lambda$ and $\lambda'$ the  values  of the parameter $t$ given by    
$$
   \lambda:= \underset{t \in \mathcal O}{\text{arg min }} L(t) \quad {\rm and} \quad   \lambda':= \underset{t \in \mathcal O}{\text{arg min }} L'(t) 
 $$
and we  set  $\rho:=L(\lambda)$ and $\rho':=L'(\lambda')$. 
  If $\mathcal O$ is large enough, the  derivatives of  $L$ and $L'$ vanish  respectively at $\lambda$ and $\lambda'$. Notice that $\rho, \rho' \in ]0, 1]$ and $\rho=1$ (resp. $\rho'=1$) if and only if $\mathbb E[\xi_1]=0$ (resp. $\mathbb E[\xi'_1]=0$).  In the cases $\bf (N, P)$ and $\bf ( P, P)$,   parameters $\rho, \rho'$ and also $\rho_\mathcal X$  are strictly less than 1 (see Theorem \ref{theo2} below).
 
These values of $\lambda$ and $\lambda'$ are useful to control  the behavior of both classical random walks $(\xi_1+\cdots +\xi_n)_{n \geq 1}$ and $(\xi'_1+\cdots +\xi'_n)_{n \geq 1}$. However, for the oscillating walk, the two distributions are strongly intertwined   and give rise in some cases to another interesting value of parameter $t$ defined by using the two functions $L$ and $L'$ considered simultaneously (see cases {\bf B1}, {\bf B4} and {\bf B7} below). More precisely, when $\lambda, \lambda' <0$  (or $\lambda, \lambda'>0$), the   graphs of $L$ and $L'$ may intersect  between $\lambda$ and $\lambda'$ at some point $ \lambda_\star$; in this case, the real  $ \rho_\star:=    L(\lambda_\star)= L'(\lambda_\star)$ 
is strictly greater than $\max (\rho, \rho')$ and ``control''   the behavior   of $\mathbb P_x[X_n=y]$ (see below the statement of Theorem \ref{theo2}).

 Note that, up to an exchange of $\mu$ and $\mu'$ or  a  symmetry with respect to the origin,  the cases $({\bf P, Z})$ and $({\bf Z, N})$ (similarly  $({\bf Z, P})$ and $({\bf N, Z})$ or  $({\bf P, P})$ and $({\bf N, N})$)    are analyzed in  an analogous  way. Therefore the cases $({\bf Z, N}), ({\bf N, Z})$ and  $({\bf N, N})$ are not considered in the present paper.   We also introduce some notations to  lighten statements. 

 \noindent {\bf Notations}
{\it For any two non negative real functions $f$ and $g$, we write
\begin{itemize}
    \item $f \preceq g$ \quad  if there is a strictly positive constant $c$  such that  $f(x) \leq c \  g(x)$.
    \item $ f(x) \sim g(x) \text{ as } x \to +\infty$ \quad if $\displaystyle \lim_{x \to +\infty} \dfrac{f(x)}{g(x)}=1.$
\end{itemize}
} 

 Here are our main statements. 
 
 \begin{theo}\label{theo1}   \underline{ Case where $\rho_\mathcal X =1$} 

Assume hypotheses {\bf O1 -- O4}   and   that   the above moment assumptions hold,  according to each situation.   Then the following statement are true

   $\bullet$ \underline{Positive recurrent case ${\bf (P, N)}$} : the Markov chain $(X_n)_{n \geq 0}$ is irreducible and aperiodic and      for any $x, y \in \mathbb Z$,
\[
 \lim_{n \to +\infty} \mathbb P_x[X_n=y] = \nu(y),
 \]
 where $\nu$ is the unique  invariant  probability measure of $\mathcal X$.
  
  $\bullet$ \underline{Null recurrent cases ${\bf (Z, Z)}$  and   ${\bf (P, Z)}$} :  for any $y \in \mathbb Z$, there exists  a strictly positive constant $C_y $    such that for any $x \in \mathbb Z$,
\[
  \mathbb P_x[X_n=y] \sim \dfrac{C_y}{ \sqrt{n}} \quad as \ \ n \to +\infty.
\]
 More precisely, if $\lambda_{\mathcal X}$ denotes the unique (up to a multiplicative constant) invariant measure for $\mathcal X$, the quantities $\displaystyle \lambda_{\tiny\mathcal X} (-\infty):= \lim_{y \to -\infty}\lambda_{\mathcal X} (y)$ and $\displaystyle \lambda_{\mathcal X}  (+\infty):= \lim_{y \to +\infty}\lambda_{\mathcal X} (y)$ do exist and 
\begin{equation} \label{jghqvelt}
C_y= \left\{\begin{array}{ll}
 \lambda_{\mathcal X}(y)\slash\sqrt{\frac{\pi}{2}}(\sigma \lambda_{\mathcal X} (-\infty)+\sigma' \lambda_{\mathcal X} (+\infty)) & \text{in the case {\bf (Z, Z)}}
 \\
 \lambda_{\mathcal X}(y)\slash\sqrt{\frac{\pi}{2}} \sigma' \lambda_{\mathcal X} (+\infty) & \text{in the case {\bf (Z, P)}}
\end{array}
\right. .
\end{equation}
$\bullet$ \underline{Transient case  ${\bf (Z, P)}$} :   for any $x, y \in \mathbb Z$, there exists a strictly positive constant $C_{x, y}$ such that  
$$  \mathbb P_x[X_n=y] \sim  \dfrac{C_{x, y}}{n^{3/2}} \quad as \ \ n \to +\infty.$$   
\end{theo}

 Notice that, in the case of the classical centered random walk on $\mathbb Z$, formula  \eqref{jghqvelt} is simplified to $C_y=  1/\sigma \sqrt{2\pi}$ as expected since $\lambda_{\mathcal X}$ is the counting measure on $\mathbb Z$

 In the   remaining transient cases $\bf (N, P)$ and $\bf ( P, P)$, the spectral radius  $\rho_\mathcal X  $ is strictly less than 1; its value depends  not only on $\rho$ and $\rho'$, which are strictly less than 1 in  both cases, but also in a quite complicated way on the relative position of the  graphs of $L, L'$ and  the Laplace transform of $\eta_1$. In the present paper, we restrict our attention to  the simpler case where the random walk oscillates between $\mathbb Z^{-}$ and $\mathbb Z^{*+}$. This relaxation purposely simplifies our analysis since only the graphs of $L$ and $L'$ are considered.   The case  where $\mu_0= \mu'$ can be treated similarly.

\begin{theo}\label{theo2}  \underline{ Case where $\rho_\mathcal X <1$} 

Assume hypotheses {\bf O1 -- O4}   and   that   the  moment assumptions hold,  according to the corresponding case.   Then,  for any $x, y \in \mathbb Z$, there exists a strictly positive constant  $C_{x, y} $ such that  $$  \mathbb P_x[X_n=y] \sim C_{x, y} \ p_n$$ where
 
\hspace{1cm} $\bullet$  Case     ${\bf (N, P)}$:   
   \qquad $\displaystyle 
p_n= \dfrac{\max(\rho, \rho')^n}{n^{3/2}}; 
  $ 

\vspace{3mm}

 \hspace{1cm} $\bullet$  Case    ${\bf (P, P)}$: 
$ \quad  
  p_n \in \left\{\dfrac{\max(\rho, \rho')^n}{\sqrt{n}},\ \dfrac{\max(\rho, \rho') ^n}{n^{3/2}},\ \rho_\star^n\right\}$

 \noindent with  $  \rho_\star= L(\lambda_\star)=L'(\lambda_\star)\in (\max(\rho, \rho'), 1) $ where $\lambda_\star$ is the unique solution   of the equation $L(t)=L'(t)$  in the interval $(\lambda, \lambda')$. 

 More precisely, it holds

 {\bf Case  A} :   $\lambda = \lambda'$ \qquad \qquad
\begin{tabular}{|c|c|c|}
\hline
\ & {\small Condition}   &  {\small   $p_n$ }
 \\
      \hline \hline 
  {\bf A1} & $\rho=\rho'$  &   {\small $   \  \rho ^n/ \sqrt{n}  $}
 \\
      \hline \hline
      {\bf A2} & $ \rho\neq  \rho'$
&     {\small $     \  \max(\rho, {\rho'})^n /  n^{3/2} $}
\\
   \hline
   \end{tabular} 

       {\bf Case  B} :   $\lambda < \lambda'$
\qquad \qquad 
   \begin{tabular}{|c|c|c|}
   \hline
\ & {\small Condition}   &  {\small   $p_n$ }
 \\
\hline \hline 
{\bf B1}  & $\rho=\rho'$ &  {\small $    \rho_\star ^n   $}
 \\
  \hline \hline
  {\bf B2} &$\rho<\rho'  \  \& \ L(\lambda')< \rho'$   &{\small $  {  \rho'} ^n /n^{3/2}   $}
            \\
      \hline
  {\bf B3} &$\rho<\rho'  \  \& \ L(\lambda')=\rho'$   &{\small $     \  {\rho'} ^n /  \sqrt{n}    $}
  \\       
            \hline
  {\bf B4} &$\rho<\rho'  \  \& \ L(\lambda')> \rho'$   &{\small $        \rho_\star^ n   $}
            \\
   \hline \hline
       {\bf B5} &$\rho>\rho'  \  \& \ L'(\lambda)< \rho$   &{\small $   \rho  ^n /n^{3/2}   $}
            \\
      \hline
  {\bf B6} &$\rho>\rho'  \  \& \ L'(\lambda)=\rho$   &{\small $      \  \rho ^n /  \sqrt{n}    $}
  \\       
            \hline
  {\bf B7} &$\rho>\rho'  \  \& \ L'(\lambda)> \rho$   &{\small $       {  \rho_\star^n}   $}
            \\
   \hline
   \end{tabular}


   {\bf Case  C} :   $\lambda  > \lambda'$
\qquad \qquad 
   \begin{tabular}{|c|c|c|}
   \hline
\ & {\small Condition}   &  {\small   $p_n$ }
 \\
\hline \hline 
{\bf C}  &  $\rho=\rho'$\ {\small or} \ $\rho\neq\rho'$ &       $\max(\rho, \rho') ^n  /n^{3/2} $ 
 \\
  \hline 
    \end{tabular}
\end{theo}
The paper is organized as follows. In section \ref{sectionordinaryrandomwalks}, we recall general results on oscillations of classical  random walks on $\mathbb Z$.  
 In section \ref{subprocess}, we introduce and study  the ``switching subprocess" of $\mathcal X$, which controls the transitions between two different media. Then, we state a general renewal theorem  due to S. Gou\"ezel,  concerning Markov walks on $\mathbb N$ with    non iid  increments; we also propose an extension of  his statement  which is useful to explore the transient cases.
Sections \ref{sectionrecurrent} and  \ref{sectiontransientZP} are devoted to the proof    of Theorem \ref{theo1}.  In section \ref{sectiontransientNPPP},  we focus on Theorem  \ref{theo2},     restricting ourselves to  the case with two media, ie  the distribution $\mu_0$ coincides with $\mu$. This restriction  stems from  the fact that in this section, the intertwining  of  two media gives rise to a number of subcases to be explored,  while a three-media  picture remains unclear. 


\section{On classical  random walks, their ladder times and  their fluctuations} \label{sectionordinaryrandomwalks}
 We denote by $S=(S_n)_{n \geq 0}$  and $S'= (S'_n)_{n \geq 0}$ the random walks defined by  $S_0=S'_0=0$ and 
 $S_n= \xi_1 + \ldots+  \xi_n, S'_n= \xi'_1 + \ldots+  \xi'_n $ for $n \geq 1$.

 Let  $ \tau^S_{+} := \inf\{k \geq 1  \mid  S_k \geq 0\}, \tau^S_{*+} = \inf\{k \geq 1  \mid  S_k > 0\}, \tau^S_{-} := \inf\{k \geq 1  \mid  S_k \leq 0\}$ and  $ \tau^S_{*-} = \inf\{k \geq 1  \mid  S_k < 0\}   $  be respectively the first {\it weakly ascending}, {\it strictly ascending}, {\it weakly descending} and  {\it strictly descending}  times  of $S$ 
(with the convention $\inf \emptyset = +\infty$).

Similarly, we defined the    ascending and descending times   $ \tau^{S'}_{+},  \tau^{S'}_{*+},  \tau^{S'}_{-} $ and $ \tau^{S'}_{*-}   $ associated with  the random walk $S'$. All these random variables are of interest to describe the oscillations of  $S$ and $S'$.

 \subsection{The centered case :  \texorpdfstring{$ \mathbb E[  \xi_1 ]=  0$}{E}}
 In this case,   it holds that  
 $\displaystyle \mathbb{P}[\liminf_{n \to +\infty}S_n=-\infty]= \mathbb{P}[\limsup_{n \to +\infty} S_n=+\infty] =1$;  hence,  the random variables $\tau^S_{+}, \tau^S_{*+}, \tau^S_{-}$  and $\tau^S_{-}$ are  $\mathbb{P}$-as  finite.  We denote respectively by 
  $\mu_{+}, \mu_{*+}, \mu_-$ and $\mu_{*-}$ the  respective distributions of  $S_{\tau_+}, S_{\tau_{*+}},  S_{\tau_-}$  and $S_ {\tau_{*-}}$   and  by  $ U_{+}, U_{*+}, U_{-}$ and  $U_{*-}$ the  corresponding potential:
\[
 U_+:= \displaystyle \sum_{n \geq 0} (\mu_+)^{\star n}, \quad  U_{*+}:= \displaystyle \sum_{n \geq 0} (\mu_{*+})^{\star n}, \quad U_-:= \displaystyle \sum_{n \geq 0} (\mu_-)^{\star n}  \quad {\rm and} \quad  U_{*-}:= \displaystyle \sum_{n \geq 0} (\mu_{*-})^{\star n}.
\] 
We also denote by $V_+, V_{*+}, V_- $  and $V_{*-}$ the  {\it weakly/strictly ascending/descending renewal functions}   of $S$  which vanish on $\mathbb Z^-$ and 
 defined   for any $x \geq 1$ by
\[
V_+(x)=  U_+
[0, x[, \quad V_{*+}(x)=  U_{*+}[0, x[, \quad V_-(x)=  U_-]-x, 0]\quad {\rm and}   \quad V_{*-}(x)=  U_{*-}]-x, 0].
\]
%
Similarly, we define renewal functions  $V'_+, V'_{*+}, V'_-$ and $ V'_{*-}$ associated with $S'$.
All these functions are increasing and   less than $C \vert x\vert$ for some strictly positive constant  $C$.    They appear crucially in the quantitative estimates of the fluctuations of $S$  and  $S'$.   
 
 The random walk $S$ corresponds to the excursions of the oscillating random walk $\mathcal X$ in $\mathbb Z^{*-}$; it is thus  natural to consider  the following stopping  times $\tau^S(x)$ with  $x \leq -1$ defined by 
$$
\tau^S(x):= \inf\{n \geq 1 \mid x+ S_n \geq 0 \}.
$$ 
 In the following lemma, we summarize classical statements  about fluctuations of random walks.  We refer to   \cite[Theorem 2, Corollary 3, Proposition 12]{doney2012}  for more details.  We emphasize on the correspondence between  $V_{*+}$ and $V_-$ in the following statements.

\begin{lemma}\label{MAJ+asym}      Assume $\sigma^2:= \mathbb E[\xi_1^2]<+\infty, \mathbb E[\xi_1]=0$ and that $\mu$ is strongly aperiodic.  

\noindent Set $   c= \frac{1}{\sigma \sqrt{2\pi}} \displaystyle \sum_{w \geq 1}V_-(w) \mu[w, +\infty).$
Then, for any $x \leq -1$, it holds that  
\begin{enumerate}[a)]
\item $\mathbb{P}[\tau^{S}(x)>n] \preceq\, \frac{1+ \vert x\vert }{\sqrt{n}}$   and  $\mathbb{P}[\tau^{S}(x)>n] \thicksim 2c \,\frac{V_{*+}(\vert x\vert )}{\sqrt{n}};$
\item   for any $y \leq -1, \quad \mathbb{P}[\tau^{S}( x)>n, x+S_n=y] \preceq\,  \frac{ (1+\vert x\vert)   (1+\vert y\vert )}{n^{3/2}}
$   and 

\hspace{3cm} $\mathbb{P}[\tau^{S}(x)>n, x+S_n=y] \thicksim\frac{1}{\sigma\sqrt{2\pi}}\, \frac{V_{*+}(\vert x\vert ) \,V_-(\vert y\vert )}{n^{3/2}}$;  
    \item $\mathbb{P}[\tau^{S}(x)=n] \preceq  \frac{ 1+\vert x\vert }{n^{3/2}}$   and 
$\quad \mathbb{P}[\tau^{S}(x)=n] \thicksim c\, \frac{V_{*+}(\vert x\vert )}{n^{3/2}}$;\  \ 
    \item for any $y \geq 0,  \quad \mathbb{P}[\tau^{S}( x)=n,  x+S_n =y] \preceq \frac{(1+\vert x\vert )}{n^{3/2}}\sum_{z> y} z\mu(z)$  and 
    
    \hspace{3cm}  $\mathbb{P}[\tau^{S}(x)=n,  x+S_n =y] \thicksim   
 \frac{1}{\sigma \sqrt{2\pi}} \frac{V_{*+}(\vert x\vert)}{n^{3/2}}\displaystyle\sum_{w \geq 1} V_-(w) \mu( w+y)$.
\end{enumerate}
\end{lemma}  
\begin{remark}  Notice that, by a classical direct argument based on Wiener-Hopf factorization \cite[Theorem II.6]{lepagepeigne}, it holds
\begin{equation} \label{kjzghv}
2c=\frac{1}{\sqrt{\pi}} \exp\left(\sum_{n=1}^{+\infty} \frac{1}{n}(\mathbb P[S_n\leq 0]-1/2)\right)=   \frac{\sigma}{\sqrt{2\pi} \mathbb E[S_{\tau^S_-}]}
\end{equation}
where the last inequality follows from \cite[Chapter XVIII.5]{feller}. 
\end{remark} 
\begin{remark} Notice that  the  inequality   $\mathbb{P}[\tau^{S}( x)>n, x+S_n=y] \preceq\,  \frac{ (1+\vert x\vert)   (1+\vert y\vert )}{n^{3/2}}$ appears in \cite[Corollary 3]{doney2012} under the restrictive assumption $\vert x\vert, \vert y\vert  \preceq \sqrt{n}$ ;  it is valid in fact without this condition. The proof is developed in \cite{dw2015} and   based on the decomposition of the event $[\tau^{S}( x)>n, x+S_n=y] $ into 3 parts. Indeed, from an easy observation, it follows  that
\begin{align*}
    [\tau^S(x)>n, \ x+S_n=y] &\subset \underbrace{[\tau^S(x)>\lfloor n/3\rfloor] }_{=: A_n} 
  \cap \underbrace{[S_{\lfloor 2n/3\rfloor}-S_{\lfloor n/3\rfloor}=y-x -S_{\lfloor n/3\rfloor}-(S_n-S_{\lfloor 2n/3\rfloor})}_{=:B_n}\\
       & \qquad \qquad  \qquad \qquad   \cap \underbrace{[y-\xi_n\leq 0, \ldots, y- (\xi_n+\cdots \xi_{\lfloor 2n/3\rfloor+1})\leq  0]}_{=: C_n}.
\end{align*}
Let $\mathfrak{A}_n$  and $\mathfrak C_n$ be the independent $\sigma$-algebras generated respectively by  the random variables   $\xi_1, \ldots, \xi_{\lfloor n/3\rfloor}$ and $\xi_{\lfloor 2n/3\rfloor+1}, \ldots, \xi_{n}$.    
From the optimal form of the discrete local limit theorem on $\mathbb R$ (see \cite[Theorem 1]{Rosen}), we obtain
\begin{align*} 
    \mathbb P[B_n \mid \mathfrak{A}_{n}\vee \mathfrak C_n] \leq  \sup_{z \in \mathbb Z} \mathbb P[S_{\lfloor 2n/3\rfloor-\lfloor n/3\rfloor}=z]\preceq \dfrac{1}{\sqrt{n}}.
\end{align*}
Hence,  $\displaystyle \mathbb P [\tau^S(x)>n, \ x+S_n=y]  \preceq \dfrac{\mathbb  P[A_n \cap B_n]}{\sqrt{n}}=  \dfrac{\mathbb  P[A_n] \mathbb P[B_n]}{\sqrt{n}}\preceq \dfrac{(1+x)(1+y)}{n^{3/2}}.$ 
  \end{remark}
We have similar results for the random variable   $\tau^{S'}(x):= \inf\{n \geq 1 \mid x+ S'_n \leq 0 \}$ when $x \geq 1$ :
\begin{align} \label{jvzhgqf}
 \mathbb{P}[\tau^{S'}(x)>n]  &\preceq  \frac{1+ x}{\sqrt{n}} \ {\rm and} \ \mathbb{P}[\tau^{S'}(x)>n] \thicksim 2c' \,\frac{V'_{*-}(x)}{\sqrt{n}}  
 \\
 & \hspace{3.5cm} {\rm with }\   c'= \frac{1}{ \sigma' \sqrt{2\pi}} \sum_{w \leq -1}V'_+(\vert w\vert) \mu'(-\infty, w], \notag\\
\mathbb{P}[\tau^{S'}(x)=n] &\preceq  \frac{1+x}{n^{3/2}}\ {\rm and}\  \mathbb{P}[\tau^{S'}(x)=n] \thicksim c'\, \frac{V'_{*-}(x)}{n^{3/2}},
 \notag
\end{align}
and  for any  $x \geq 1$ and $y \leq 0$,
 \begin{align}\label{functiong}
 \mathbb{P}[\tau^{S'}(x)=n, x&+S'_n=y]  \preceq\,  \frac{(1+x)}{n^{3/2}}\sum_{z <y} \vert z\vert \mu'(z)  \notag
\\
& {\rm and }  \ 
\mathbb{P}[\tau^{S'}(x)=n, x+  S'_n  =y] \thicksim \frac{V'_{*-}(x)}{\sigma' \sqrt{2\pi} n^{3/2}}\sum_{w \leq -1} V'_+(\vert w\vert) \mu'(  w+y).
\end{align}
 
 Similarly, as above,  by \cite[Theorem II.6]{lepagepeigne} and \cite[Chapter XVIII.5]{feller}, it holds
\begin{equation} \label{jhevgt}
2c'=\frac{1}{\sqrt{\pi}} \exp\left(\sum_{n=1}^{+\infty} \frac{1}{n}(\mathbb P[S'_n\geq 0]-1/2)\right)= \frac{\sigma'}{\sqrt{2\pi} \mathbb E[S'_{\tau^{S'}_+}]}.
\end{equation}

\subsection{The non centered case}

In order to control   the transitions  between two different media in cases $( \bf Z, P), (\bf N, P)$ and $(\bf P, P)$, we also need upper bounds of the quantities $    \mathbb P[\tau^S(x)>n] $ and    $ \mathbb P[\tau^S(x)>n, \ x+S_n=y]$ when $\mathbb E[\xi_1]>0$ and the corresponding quantities for $S' $ when $\mathbb E[\xi'_1]<0$.  

When  $\mathbb E[\xi_1 ] $ is strictly positive,  the  random variables  $\tau_+^S$    has finite expectation,  and even moment of order 2 as soon as $\mathbb E[(\xi_1^-)^2]<+\infty$, see \cite[Theorem 3.1]{gut}. The same  property  holds for the variables  $\tau^S(x)$ for $ x \leq -1$. We did  not find in the literature a control of the probability $ \mathbb P[\tau^S(x)>n]$ in the case where $\mathbb E[\xi_1 ] $ is strictly positive; this is the aim of the next lemma under a  polynomial moment assumption.

     \begin{lemma}\label{estimate drift}
         Assume that $\mathbb E[\xi_1^2]<+\infty$ and $\mathbb E[\xi_1] >0$. Then, there exists a strictly positive constant $C$ such that for any $x \leq  -1$, 
    \begin{equation}\label{momenttau}
         \mathbb E[(\tau^{S}(x) )^2]\leq C (1+ \vert x\vert )^2.
       \end{equation}
         Hence, for any $\varepsilon  \in [0, 1]$  there exists  a strictly positive constant $C_\varepsilon $ such that  for any $x\leq -1$ and  $n \geq 1$, 
         \begin{align}\label{tauS>n}
   \mathbb P[\tau^S(x)>n] \leq  C_\varepsilon \dfrac{ (1+\vert x\vert )^{1+\varepsilon}} {n^{1+\varepsilon}} 
   \end{align}
and 
         \begin{align}\label{tsejvhtb}
    \mathbb P[\tau^S(x)=n, \ x+S_n=y]   \leq  C_\varepsilon \dfrac{ (1+\vert x\vert)^{1+\varepsilon}}{n^{{\frac{3}{2}}+\varepsilon}}\ \mu(y, +\infty)\quad {\rm when} \quad  y \geq 0. 
    \end{align} 
     \end{lemma}         
\begin{proof} \underline{Let us first prove inequalities \eqref{momenttau} and \eqref{tauS>n}}. Since $\mathbb E[\xi_1]$ is strictly positive, we may  consider  the sequence $(\tau_k)_{k \geq 0}$ of strictly ascending times of $S$ defined inductively by :
\[
\tau_0=0\quad {\rm and} \quad  \tau_{k+1}= \inf \{n > \tau_k\mid S_n>S_{\tau_k}\}.
\]
The sequence $(\tau_{k+1}-\tau_{k})_{k \geq 0}$ is iid.
 We observe that, for any $x \leq -1$,
\begin{align} \label{tau_expectation}
        \mathbb E[(\tau^S(x))^2] &\leq   2  \sum_{n \geq  1} n\ \mathbb P[\tau^S(x)\geq n] \notag\\
        &\leq  2  \sum_{n \geq 1} n  \sum_{k=0}^{\vert x\vert} \mathbb P[\tau_k \leq n < \tau_{k+1}, S_n \leq \vert x\vert] 
       \quad {\rm since} \    S_{\tau_k} \geq k,\ \notag \\
         &\leq  2 \sum_{k=0}^{\vert x\vert} \sum_{n \geq k} n\ \mathbb P[\tau_k \leq n < \tau_{k+1}] \quad {\rm since} \ \tau_k \geq k.
    \end{align}
The inner sum of the right side term in \eqref{tau_expectation} may be controlled as  follows.
    \begin{align*}
         \sum_{n \geq k} n \ \mathbb P[\tau_k \leq n < \tau_{k+1}] &=  \sum_{n \geq k} n  \sum_{\ell \geq 1} \mathbb P[\tau_{k+1}-\tau_k=\ell, \tau_k \leq n <  \tau_k+\ell] \\
        &=  \sum_{n \geq k} n \sum_{\ell \geq 1} \mathbb P[\tau_{k+1}-\tau_k=\ell]\mathbb P[n-\ell+1  \leq \tau_k \leq n] \\
        &=  \sum_{\ell \geq 1} \mathbb P[\tau^S_{*+}=\ell]  \underbrace{\sum_{n \geq k}  n\sum_{i=n-\ell+1 }^n \mathbb P[\tau_k=i]}_{\leq \displaystyle  \sum_{i \geq k-\ell+1}  \mathbb P[\tau_k=i] \sum_{n= i}^{i+\ell-1} n}  \   \\
        & \leq \sum_{\ell \geq 1} \mathbb P[\tau^S_{*+}=\ell]  \sum_{i\geq 1 } (i\ell +\ell^2)\mathbb P[\tau_k=i]
        \\ 
        & =\mathbb E[\tau^S_{*+}] \mathbb E[\tau_k]+ \mathbb E[(\tau^S_{*+})^2] \  \leq (k+1)\mathbb E[(\tau^S_{*+})^2] 
    \end{align*} 
 Hence  $ \mathbb E[(\tau^S(x))^2]    \preceq  (1+\vert x\vert)^2 $. 
        The upper bound \eqref{tauS>n} follows by  using the Markov and Jensen inequalities: for any $\varepsilon \in [0, 1]$,

    \begin{align*}
      \mathbb P[\tau^S(x)>n] &\leq \dfrac{\mathbb E[(\tau^S(x))^{1+\varepsilon}]} {n^{1+\varepsilon}}
       \leq 
        \dfrac{\mathbb E[(\tau^S(x))^2]^{\frac{1+\varepsilon}{2}}}{n^{1+\varepsilon}}
        \leq
      C_\varepsilon \dfrac{(1+\vert x\vert)^{1+\varepsilon}}{n^{1+\varepsilon}}.
    \end{align*}
    
\noindent \underline{Proof of inequality \eqref{tsejvhtb}}. We first control the quantity 
$\mathbb P[\tau^S(x)>n, \ x+S_n=z]$ for any $z\leq -1$  by  decomposing the event $[\tau^S(x)>n, \ x+S_n=z]$ as
\begin{align*}
    [\tau^S(x)>n, \ x+S_n=z] \subset \underbrace{[\tau^S(x)>\lfloor  n/2\rfloor]}_{E_{n, x}} \cap [S_n-S_{\lfloor  n/2\rfloor}=z-x -S_{\lfloor  n/2\rfloor}].
\end{align*}
Let $\mathfrak{G}_{\lfloor  n/2\rfloor}$ be the $\sigma$-algebra generated by the random variables $\xi_1, \ldots, \xi_{\lfloor  n/2\rfloor}$. Obviously, the event $E_{n, x}$ and the random variable $S_{\lfloor  n/2\rfloor}$ are both $\mathfrak{G}_{\lfloor  n/2\rfloor}$-measurable, whereas $S_n-S_{\lfloor  n/2\rfloor}$ is independent of $\mathfrak{G}_{\lfloor  n/2\rfloor}$.
 It is straightforward from  \cite[Theorem 1]{Rosen} that 
\begin{align} \label{local limit theorem on R}
    \mathbb P[S_n-S_{\lfloor  n/2\rfloor}=z-x-S_{\lfloor  n/2\rfloor} \mid \mathfrak{G}_{\lfloor  n/2\rfloor}] \leq  \sup_{w \in \mathbb Z} \mathbb P[S_{n-\lfloor  n/2\rfloor}=w]\preceq \dfrac{1}{\sqrt{n}}.
\end{align}
Hence, by \eqref{tauS>n} 
 \begin{align*}
 \mathbb P[\tau^S(x)>n, \ x+S_n=z] &\leq \mathbb P[E_{n, x}] \ \sup_{w \in \mathbb Z} \mathbb P[S_{n-\lfloor  n/2\rfloor}=w]\leq  C_\epsilon  \dfrac{(1+\vert x\vert)^{1+\varepsilon}}{ n^{{3\over 2}+\varepsilon}}
 \end{align*}
 with $C_\varepsilon >0$.
 Consequently, for $y \geq 0$, one gets
 \begin{align*}
     \mathbb P[\tau^S(x)=n, \ x+S_n=y]   &= \sum_{z\leq -1}  \mathbb P[\tau^S(x)>n-1, \ x+S_{n-1}=z]\
     \mu(\vert z\vert +y)\\
     &\leq  C_\epsilon  \dfrac{(1+\vert x\vert)^{1+\varepsilon}}{ n^{{3\over 2}+\varepsilon}}\ \mu(y, +\infty) 
 \end{align*}
 which completes the proof of  \eqref{tsejvhtb}.
 \end{proof} 
In order to get a  precise asymptotic  of $\mathbb P[\tau^S(x)=n, \ x+S_n=y]  $ as $n \to +\infty$, we require   the existence of sufficiently large exponential moments for $\xi_1$. This relies on  a classical argument in the theory of random walks, by   returning to the centered case with a change of measure. As explained in section \ref{statement}, under hypotheses {\bf O1}, {\bf O2} and {\bf O3}, the Laplace transforms $L$ and $L'$ of $\mu$ and $\mu'$ are strictly convex on their respective  domain  of definition and  reach their minimum  at  $\lambda$ and $\lambda'$ respectively; consequently,     the probability  measures  $ {\mu}^{\lambda} ({\rm d}x):= e^{\lambda x} \mu( {\rm d}x)/ L(\lambda)$   and  $ {\mu'}^{\lambda} ({\rm d}x):= e^{\lambda' x} \mu'( {\rm d}x)/ L'(\lambda')$   are centered on $\mathbb Z$. 

As a consequence, if $ \overset{\circ}{S_{\ }} $ denotes the centered random walk associated with $S$ and with increment distribution $ { \mu}^{\lambda}$, then the distribution of $  {\overset{\circ}{S_n}}$ for $n \geq 1$ is $ ({ \mu}^{\lambda}\,)^{\star n}  ({\rm d}x)= e^{\lambda x}  \mu\,^{\star n} ( {\rm d}x)/ \rho\,^n$. More precisely,    for any  function $\psi: \mathbb Z^n \to \mathbb R^+$,   it holds that  
$$
\mathbb E[\psi(S_1, \ldots, S_n)]= \rho\,^n   \mathbb E[\psi(\overset{\circ}{S_1}, \ldots, \overset{\circ}{S_n})e^{-\lambda\overset{\circ}{S_n}}].  
$$

Combining  \cite[Theorem 2.1]{iglehart} and   Lemma \ref{MAJ+asym}, we obtain the following lemma (with  a similar statement   for the random walk $S'$ when $\mathbb E[\xi'_1]\neq 0$).
\begin{lemma} \label{relativisation}Assume that $\mathbb E[\xi_1]\neq 0$ and that 
  there exists $\lambda\neq 0$ such that 
 $$
  \mathbb E[ e^{\lambda \xi_1}] <+\infty, \   \  \mathbb E[\xi_1 e^{\lambda \xi_1}]=0 \   and  \  \ 
 \mathbb E[ \xi_1^2 e^{\lambda  \xi_1}] <+\infty.  
  $$
  Then,   $\rho :=  \mathbb E[ e^{\lambda \xi_1}]<1$ and  the random walk ${\overset{\circ}{S_{\ }} }$ with distribution increment $\mu^{\lambda}( {\rm d}x)= e^{\lambda x} \mu( {\rm d}x)/ \rho$ is centered. Let $ \overset{\circ}{\sigma}   := \sqrt{\mathbb E[\xi_1^2e^{\lambda \xi_1}]/\rho}$ be the standard deviation of  ${\overset{\circ}{S_{\ }} }$ and  denote by $ \overset{\circ}{V}_{*+}$ (resp. $  \overset{\circ}{V}_-$)   its strictly ascending (resp. weakly descending)  renewal  function. 
  
Furthermore, for any $x \leq -1$ and $y \geq 0$,
\begin{align*}
 \mathbb{P}[\tau^{S}(x)=n, x+S_n=y] 
 &=  \rho\,^n \ e^{\lambda(x-y)} \mathbb P[\tau^{\overset{\, \circ}{S}}(x)= n, x+\overset{\circ}{S_n}= y]
 \\
 &
 \preceq\,  (1+\vert x\vert) e^{\lambda x} \left(e^{-\lambda y} \sum_{z>y} z e^{\lambda z}\mu (z)\right) \frac{ \rho^{n }}{n^{3/2}}
\end{align*}
and 
\begin{equation}\label{valeurrelativisation}
\mathbb{P}[\tau^{S}(x)=n,  x+S_n= y] \thicksim\frac{1}{\overset{\circ}{\sigma} \sqrt{2\pi}}\, \overset{\circ}{V}_{*+}(\vert x\vert) e^{\lambda x} \left(\sum_{w \geq 1}\overset{\circ}{V}_-(w)e^{\lambda w} \mu (w+y)\right)\frac{\rho^{n-1}}{n^{3/2}}. 
 \end{equation}
 \end{lemma}

We  end  this section with an oversestimation of the quantity $\mathbb P[\tau^{S'}(x) =n, x+S'_n=y]$ with $x\geq  1$ and $y\leq  0$   when $\mathbb E[\xi'_1] >0$, it is useful in the analysis of the case $(\bf Z, P)$. When  $\mathbb E[e^{t \xi'_1}]<+\infty$ for some $t<0$, we may consider     the random walk 
$({S'_n}^{t})_{n \geq 0}$ with jump distribution 
${\mu'}^{ t}({\rm d}x)= e^{ tx}\mu'({\rm d}x)/L'(t)$ and write 
\begin{align*}
\mathbb P[\tau^{S'}(x) =n, x+S'_n=y]=
 L'(t)^n e^{t (x-y)} \mathbb P[\tau^{S'}(x)=n, x+{S'_n}^{t}=y]\leq  L'(t)^n e^{t (x-y)}
\end{align*}
with $L'(t) <1$ if $t<0$ is close to $0$. Without  exponential moments assumption, the argument is a little bit more delicate and based on Nagaev-Fuchs inequality: we   detail it here. 
 \begin{lemma}\label{jhzcabgjehf}
  Assume that $\mathbb E[(\xi'_1)^2]+\mathbb E[({\xi'_1}^-)^{p}]<+\infty$ with $p\geq 2$ and $m':= \mathbb E[\xi'_1]>0$.   Then, there exists $ C_{ p}>0$ such that, for any $x\geq  1, y \leq 0$ and $n \geq 1$,
\begin{equation}\label{qvzgbrvhk}
 \mathbb P[\tau^{S'}(x) =n, x+S'_n=y ]\leq C_p\left({1 \over n^{p-1 } (1+\vert y\vert ^{p})}+{1\over (n+\vert y\vert)^p}\right).
\end{equation}
  In particular,  for any $\epsilon \in [0, 1)$, there exists $ C_{  p, \epsilon}>0$ such that
\begin{equation}\label{zqhcgbef}
 \mathbb P[\tau^{S'}(x) =n, x+S'_n=y ]\leq {C_{p,\epsilon}\over n^{p-1-\epsilon} (1+\vert y\vert ^{1+\epsilon})}.
 \end{equation}
 \end{lemma}
  \begin{proof}
 We set $a:= m'/2$ and first notice that, by  \cite{nagaev},  there exists a strictly positive constant $c=c_{p}>0$ such that
 \begin{align*}
  \mathbb P[  S'_n<an ]&= 
    \mathbb P[  S'_n-nm'< -an  ]
  \\
 & \leq n \mathbb P[\xi'_1-m'<-an]+ c {n \over (an)^p}.
 \end{align*}
 Then $\displaystyle \mathbb P[ x+S'_n<an ]\leq {c'\over n^{p-1}  }$
 for some  $c' = c'_{p}>0$.
 As a consequence,  for $x\geq  1, y\leq 0$ and $n \geq 2$, 
 \begin{align*}
 \mathbb P[\tau^{S'}(x)=n, x+S'_n= y]&=\sum_{  z\geq 1}
  \mathbb P[\tau^{S'}(x)>n-1, x+S'_{n-1}= z]\, \mu'( -z +y)\\
  &\leq \sum_{1\leq z < an}  
  \mathbb P[\tau^{S'}(x)>n-1, x+S'_{n-1}= z]\, \mu'(-z  +y) 
+ \sum_{z\geq an} \mu'( - z  +y)
  \\
  &\leq 
    \mathbb P[S'_{n-1}< an] \ \mu'(-\infty , y)\   +  \mu'(-\infty,  y-an]
    \\ 
   &\leq  {c' \over n^{p-1} } \times {\mathbb E[({\xi'_1}^-)^p]\over 1+\vert y\vert^p}+ { \mathbb E[({\xi'_1}^-)^p]\over ( an +\vert y\vert )^p}
   \\
    & \leq   C_p \left( \frac{1}{n^{p-1 } (1+\vert y\vert^{p})}+ \frac{1}{(n+\vert y\vert)^p}\right), \quad {\rm for\  some\  strictly \ positive\  constant \ } C _{p}.
 \end{align*}
 
Let us now establish
  \eqref{zqhcgbef}. The inequality  
 $\displaystyle
 \frac{1}{n^{p-1 } (1+\vert y\vert^{p})}\preceq \frac{1}{ n^{p-1-\epsilon} (1+\vert y\vert ^{1+\epsilon})}
 $ is obvious   since  $p \geq 2$ and $\epsilon \in [0, 1)$.
 To prove that 
 $\displaystyle
\frac{1}{(n+\vert y\vert)^p}\preceq \frac{1}{ n^{p-1-\epsilon} (1+\vert y\vert ^{1+\epsilon})},
 $  
 it suffices to notice that  
 $$
  \inf_{ n, t \geq 1  } 
  \frac{(n+t)^p}{n^{p-1-\epsilon} t^{1+\epsilon} }
  \geq 
 \inf_{u>0}
 \left(u^{\frac{1+\epsilon}{ p}}  +\left(\frac{1}{u}\right)^{1-\frac{1+\epsilon}{ p}}\right)^p >0. 
  $$ 
 \end{proof}
\section{On the  switching subprocess  of   \texorpdfstring{$\mathcal X$}{X} }\label{subprocess}	  
 In this section, we introduce a subprocess of $\mathcal X$ which is of key importance in the sequel.  

 \subsection{The  switching subprocess  }
Let ${\bf C}=(C_k)_{k \geq 0} $  be the sequence of  ``switching times'',  ie  times at which the process $\mathcal{X}$    swaps regime.  We set  $C_0=0$ and  for any $k \geq 0$ such that $C_k <+\infty$, 
\begin{align}\label{ switchingtime}
    C_{k+1} :=\left\{ 
    \begin{array}{lll}
 \vspace{2mm}
    \inf \{n > C_k \mid X_{C_k} + (\xi_{C_k + 1} +  \dotsi +\xi_{n}) \geq 0\} &if \quad   X_{C_k} \leq -1, \\
      \vspace{2mm}
      \inf \{n > C_k \mid X_{C_k} + (\eta_{C_k + 1} +  \dotsi +\eta_{n}) \neq 0\}  &if \quad X_{C_k} =0, \\
 \inf \{n > C_k \mid   X_{C_k} + (\xi'_{C_k + 1} + \dotsi + \xi'_{n}) \leq 0 \} &if \quad X_{C_k} \geq 1 
 \end{array}
 \right.
\end{align} 
and $C_{k+1}=+\infty$ when $C_k= +\infty$.
  The  random variables $ C_k$ are stopping times with respect to the canonical filtration $(\mathcal F_n)_{n \geq 1}$ associated with the sequence $(\xi_{n}, \eta_n, \xi'_{n})_{n \geq 1}$.  
 When $\mathbb E[\xi_1]\geq 0$ and $\mathbb E[\xi'_1]\leq 0$,  these random  variables  are $\mathbb{P}$-as  finite. 
 When  $\mathbb E[\xi_1]< 0$ or $\mathbb E[\xi'_1]> 0$,  it holds  $C_k =+\infty$ for sufficiently large (random) $k$ ie the Markov chain  $(X_n)_{n \geq 0}$ stays ultimely   in exactly one of the half lines $\mathbb Z^{*-}$ or $\mathbb Z^{*+}$.

We denote by $\mathcal X_{\bf C}:=(X_{C_k}{\bf 1}_{[C_k<+\infty]})_{k \geq 0}$ the  subprocess of ${\bf C}$ and call it the {\bf  switching subprocess of $\mathcal X$}. It plays a crucial role  in the sequel.

 By \cite[Lemma~3.2]{peignevo}, when the sequence $(C_k)_{k \geq 0}$ is finite $\mathbb P$-as, 
 the process $\mathcal X_{\bf C}$  is a  time-homogeneous Markov chain on $\mathbb{Z}$ with transition kernel 
$ Q = (Q(x, y))_{x, y \in \mathbb Z}$ given by $Q(x, y) := \mathbb P_x[C_1<+\infty, X_{C_1}=y] $ for any $x, y \in \mathbb Z$. 
This readily implies that $Q^{(\ell)}(x, y) = \mathbb P_x[C_\ell <+\infty, X_{C_\ell}=y]$ for any $\ell \geq 0$ and yields
  the following formula
\begin{equation}\label{Qswitchingtime-finite1}
Q(x, y)  =   
 \left\{
   \begin{array}{clcll}
 \displaystyle \sum_{n \geq 1} \mathbb P [\tau^S(x)=n, x+S_n=y] &  if &x \leq -1& and &y \geq 0,  \\
  \displaystyle {\mu_0(y)\over 1-\mu_0(0)}  & if &x = 0& and &y \neq 0, 
   \\
  \displaystyle \sum_{n \geq 1} \mathbb P[\tau^{S'}(x)=n, x+S'_n=y]  & if &x \geq 1 &and &y \leq 0.   
   \end{array}
    \right.
\end{equation}
The switching times $C_\ell$ are  alsot  ascending and descending  ladder  times of $S$ and $S'$ respectively, hence we may also write
\begin{equation}\label{Qswitchingtime-finite2}
 Q(x, y) = \left\{
   \begin{array}{clcll} 
  \displaystyle\sum_{t = 0}^{\vert x\vert -1} \mu_{*+}(y-x-t)\, U_{*+}(t) &  if &x \leq -1& and &y \geq 0,\\
  \displaystyle {\mu_0(y)\over 1-\mu_0(0)} &  if &x = 0& and &y \neq 0, 
  \\
   \displaystyle\sum_{t = -x+1}^{0} \mu'_{*-}(y-x-t)\,\ U'_{*-}(t) & if &x \geq 1 &and &y \leq 0.
   \end{array}
    \right. 
\end{equation}
Expression \eqref{Qswitchingtime-finite2} is useful in the sequel to prove that $Q$ is a compact operator on some suitable  space depending on the explored situation.    Let ${\bf S}_{\mu_0}$ denote the support of $\mu_0$.
Following the argument developed in \cite{vo1}, we can prove that   the set   
\begin{equation*} \label{qzvge}
\mathcal I_{\bf C}=] D',  D[\  \cup \ {\bf S}_{\mu_0}  
\end{equation*}
 is the unique essential class for the switching process $\mathcal X_{\bf C}$. Indeed, on the one hand, since $\mu$ is aperiodic, it holds that $Q(x, y) >0$ for any $x \leq -1$ and $y \in [0, D[$ because $\mu(D)>0$ and there exists a path governed by $\mu$ from $x$ to $y-D\leq -1$  which stays in $\mathbb Z^{*-}$.  The same property holds for $x \geq 1$ and $y \in ]D', 0]$.  On the other hand  $Q(0, y)>0$ for any $y\in {\bf S}_{\mu_0}$. Since $Q(x, y)=0$ when $x \in   \mathcal I_{\bf C}$ and $y \notin \mathcal I_{\bf C}$, this proves that $\mathcal I_{\bf C}$  is irreducible.
When $\mathbb E[\xi_1]<0$, it holds that $\mathbb P[C_1<+\infty]<1$ for any $x\leq -1$ and $\mathcal I_{\bf C}$ is transient for $\mathcal X_{\bf C}$ ; the same property holds when $\mathbb E[\xi'_1]>0$. When $\mathbb E[\xi_1]\geq 0$ and $\mathbb E[\xi'_1]\leq 0$, the class $\mathcal I_{\bf C}$ is the unique essential class of  $\mathcal X_{\bf C}$ and it is recurrent.

It is natural to decompose $Q$ as $Q:= \sum_{n\geq 1} Q_n$ with 
 $Q_n(x,y):= \mathbb P_x[C_1=n, X_n=y] $ 
for any $x, y \in \mathbb Z$.  More precisely, 
\begin{equation}\label{Qswitchingtime-k}
Q_n(x, y)  =   
 \left\{
   \begin{array}{clcll}
 \displaystyle  \mathbb P [\tau^S(x)=n, x+S_n=y] &  if &x \leq -1& and &y \geq 0, 
  \\
  \displaystyle  \mu_0(0)^{n-1}\mu_0(y) &  if &x = 0& and &y \neq 0, 
  \\
  \displaystyle  \mathbb P[\tau^{S'}(x)=n, x+S'_n=y]  & if &x \geq 1 & and &y \leq 0.   
   \end{array}
    \right.
\end{equation}
Similarly,  for any $\ell \geq 0$, the $\ell$-th power  $ Q^{(\ell)}$ of $Q$    may also be decomposed as 
 $ Q^{(\ell)}:=\displaystyle  \sum_{n\geq 1} Q^{(\ell)}_n$ where 
\begin{align*}
  Q^{(\ell)}_n(x, y)&= \mathbb P_x[C_{\ell}=n, X_n=y]\\
  &=
\sum_{\stackrel{j_1, \ldots, \ j_\ell \geq 1  }{ j_1+\ldots +j_\ell = n}} \sum_{y_1, \ldots, y_\ell=y}  \mathbb P_x[C_1=j_1, X_{j_1}=y_1]  \times \cdots  \times \mathbb P_{y_{\ell-1}}[C_1=j_\ell, X_{j_\ell}=y ]  
\\
&=\sum_{\stackrel{j_1, \ldots, \ j_\ell \geq 1  }{ j_1+\ldots +j_\ell = n}} Q_{j_1} \ldots Q_{j_\ell}{\bf 1}_{\{y\}}(x).
\end{align*}

 \subsection{Decomposition of the trajectories of  \texorpdfstring{$\mathcal X$}{X} }\label{decomposition}

 The  subprocess $\mathcal X_{\bf C}$ allows us to decompose the trajectories of $(X_n)_{n \geq 0}$ into successive excursions inside   the different media. In particular,  for any $n \geq 1, x \in \mathbb Z$ and $y \geq 1$ (similar formulae hold when  $y \leq 0$), the    probability $\mathbb P_x[X_n=y] $  may be decomposed as 
 followed. 
\begin{align}  \label{decompositiontraject}
\mathbb P_x[X_n=y]&=
\sum_{\ell=0}^{+\infty} \mathbb P_x[X_n=y, C_\ell\leq n <C_{\ell+1}]\notag
\\
& =    \sum_{\ell=0}^{+\infty} \sum_{k=0}^{n-1}\sum_{z \geq 1} \underbrace{\mathbb P_x[C_\ell=k, X_k= z]}_{= Q^{(\ell)}_k(x,z)}\ 
\underbrace{ \mathbb P[z+S'_1 \geq 1, \ldots, z+S'_{n-k-1} \geq 1, z+S'_{n-k}=y]}_{:=   {\bf V}_{  n-k, y}(z)}  \notag
\\
&
 \qquad \qquad +  \sum_{\ell=0}^{+\infty}  
\underbrace{\mathbb P_x[C_\ell=n, X_n= y]}_{= Q^{(\ell)}_n(x,y)}\  \quad{\rm with\ the \ convention} \ 
Q^{(0)}_k(x,z)= {\bf 1}_{\{x\}}(z),
\notag\\
& =  \sum_{k=0}^n   \left(\left(\sum_{\ell=0}^{+\infty} Q^{(\ell)}_k\right)  {\bf V}_{ n-k, y}\right)(x),  \ \text{with the convention} \   {\bf V}_{0,y}(z):= {\bf 1}_{\{y\}}(z).
\end{align}
In the sequel we thus consider the functions $  {\bf V}_{n, y}$ for $ n \geq 0$ and $ y \in \mathbb Z$  corresponding   to the final excursions  of the trajectories inside  the positive or negative half-lines  and defined by $  {\bf V}_{0, y}(z):= {\bf 1}_{\{y\}}(z) $ and, for $n \geq 1$,

$\bullet$ if $y\leq -1$ then 
\begin{equation}\label{hny}
   {\bf V}_{n, y}: x \mapsto \left\{\begin{array}{cll}
  \mathbb P[\tau^S(x)>n, x+S_n=y] 
 & if & x \leq -1, 
  \\
  0&if & x \geq 0,
\end{array} 
\right. 
\end{equation}

$\bullet$ if $y=0$ then $   {\bf V}_{n, 0}(0)=\mu_0(0)^n$ and $  {\bf V}_{n, 0}(x)=0$ otherwise,

$\bullet$ if $y\geq  1$ then 
\begin{equation}\label{h'ny}
  {\bf V}_{n, y}: x\mapsto\left\{\begin{array}{cll}
 0&  if & x \leq 0,  
  \\
 \mathbb P[\tau^{S'}(x)>n, x+ S'_n =y] & if & x \geq 1.
\end{array}
\right.
\end{equation}

 When  $\mathbb E[\xi_1]\geq 0$ and $\mathbb E[\xi'_1]\leq 0$, it holds  that $\mathbb P_x[C_\ell <+\infty]=1$ for any $\ell \geq 0$. In this case, we  need to control firstly the behavior  as $n \to +\infty$ of the transition operator $T_n$ defined for any $x, z \in \mathbb Z$ and $n \geq 1$ by  
\begin{equation}\label{azevb}
T_n(x, z):= \sum_{\ell=1}^{+\infty}Q_n^{(\ell)} (x, z)
=\displaystyle\sum_{\ell=0}^{+\infty}  \mathbb P_x[C_\ell= n, X_ n= z] \end{equation} 
and secondly the one of the functions $  {\bf V}_{ n, y}$. Since  $C_\ell =\displaystyle \sum_{i=0} ^{\ell-1} (C_{i+1}-C_i)$ for any $\ell \geq 1$, the control of $T_n$  as $n \to +\infty$  corresponds to a renewal theorem   for the sequence of non iid random variables $(C_{i+1}-C_i)_{i \geq 0}$  while the one of  $  {\bf V}_{ n, y}$ deals with  fluctuations of the   random walk  $S'$.

When  $\mathbb E[\xi_1]<0$  (resp. $\mathbb E[\xi'_1]> 0$),  it holds that   $\mathbb P_x[C_1 <+\infty]\leq \mathbb P_{-1}[C_1 <+\infty] <1 $ when $x \leq -1$ (resp.  $\mathbb P_x[C_1 <+\infty]\leq \mathbb P_{1}[C_1 <+\infty] <1$ when $x \geq 1$) so that the quantity  $\mathbb P_x[C_\ell <+\infty]$
 decreases exponentially  to $0$,   uniformly in $x$. In this case, we need to control the behavior of $\mathbb P_x[C_\ell=k, X_k= z] $ as $k \to +\infty$, for any $\ell \geq 1$ and $z \in \mathbb Z$. 
  
 In both  cases, the fact that the random variables  $C_{i+1}-C_i$ are not  iid  forces us to use a functional version  of a renewal theorem that we present in the following section.

   \subsection{A local limit theorem for an aperiodic  sequence of operators}
	
The previous decomposition $\displaystyle C_\ell {\bf 1}_{[C_\ell <+\infty]} = \sum_{i=0} ^{\ell-1} (C_{i+1}-C_i){\bf 1}_{[C_{i+1} <+\infty]}$ for $ \ell \geq 1 $ stipulates that $C_\ell$ is a sum of the integer valued random increments  $\Delta_i=C_{i+1}-C_{i}$.

If the  $\Delta_i $  were  iid, the process $(C_\ell)_{\ell\geq 0}$  would be  a classical random walk $(\Delta_0+ \Delta_1+\ldots + \Delta_{\ell-1})_{\ell\geq 0  }$  on $\mathbb N$ with increments $\Delta_i$ having an infinite mean, since by   Lemma \ref{MAJ+asym} the  tail distribution function of these increments  would satisfy  $\mathbb P[\Delta_i>n]\sim c / n^\beta $ with $\beta = 1/2$ and $c>0$.
 The precise asymptotic behavior as $n \to +\infty$ of the quantity $\sum_{\ell \geq 0} \mathbb P[C_\ell =n]$ would be given by 
renewal theorems, which do exist  in this infinite mean situation (we refer for instance to the seminal works by  \cite{garsialamperti} and \cite{erickson}).  Nevertheless, when $0<\beta\leq 1/2 $, the strong renewal theorem  is not valid in full generality,  but only when $n$ varies in a set of integers having density 1. This result has been improved by R.A. Doney \cite{doney1997} under the additional hypothesis $\mathbb P[\Delta_1= n] \preceq  L(n)/n^{1+\beta}$ where $L$ is a slowly varying function.

In our cases, the random variables $C_{i+1}-C_i $ are not  iid  and the strong renewal theorem is more complicated to prove. We use a result due to   Gou\"ezel to study the  cases 
 $({\bf P,N}),  (\bf Z,Z)$ and $(\bf P,Z)$. We make  precise this statement  in Theorem \ref{theogouezel+} under  some additive hypotheses   to  analyze also the cases  $({\bf Z, P}),   ({\bf N, P})$ and $({\bf P, P})$.

	  An {\it aperiodic  renewal sequence of operators} is a  sequence  $(\mathcal R_n)_{n \geq 1}$  of operators acting on a  $\mathbb C$-Banach space $({\mathcal B}, \vert \cdot \vert_{\mathcal B })$ and satisfying the following conditions.
	  
{\it 
 $\bullet$ For $n\geq 1$, the operators $\mathcal R_n$  act  on ${\mathcal B }$  and $\displaystyle \sum_{n \geq 1} \Vert \mathcal  R_n\Vert_{\mathcal B } <+\infty$ (where $\Vert\cdot \Vert_{\mathcal B }$ denotes the operator norm on the Banach space $\mathcal L({\mathcal B })$ of linear continuous operators on $({\mathcal B },  \vert \cdot \vert _{\mathcal B })$).

 $\bullet$ The operators  $\displaystyle \mathcal R(z):= \sum_{n \geq 1} z^n \mathcal R_n$  defined for any $z $ in  the  closed unit ball $\overline{\mathbb{D}}$ in $\mathbb C$ satisfy  
 \begin{enumerate}[R1-]
\item   $\mathcal R(1)$ has a simple eigenvalue at 1 (with corresponding eigenprojector $\Pi_\mathcal R$) and the rest of its spectrum is contained in a disk of radius strictly less than $ 1$;

\item for any complex number $z \in \overline{\mathbb{D}} \setminus \{1\}$, the spectral radius of $\mathcal R(z)$ is strictly less than $1$; 

\item  for any $n \geq 1$, the real number $r_n$ defined by $\Pi_\mathcal R \mathcal R_n \Pi_\mathcal R= r_n \Pi_\mathcal R$ is non negative.   
\end{enumerate} 
 }
\noindent  Condition {\it R2} implies that,  for any $ z \in \overline{\mathbb D}\setminus \{1\}$, the operator $I-\mathcal R(z)$ is invertible on ${\mathcal B }$  and
\[
(I-\mathcal R(z))^{-1}= \sum_{\ell \geq 0} \mathcal R(z)^\ell= \sum_{\ell \geq 0} \left(\sum_{j \geq 1}
z^j\mathcal R_j\right)^\ell=
\sum_{n \geq 0}  
{\mathcal T}_n z^n
\]
where    $  {\mathcal T} _0= I$ and $\displaystyle   {\mathcal T} _n:= \sum_{\ell=1}^{+\infty} \left( \sum_{j_1+ \ldots + j_\ell=n}{\mathcal R}_{j_1}\ldots {\mathcal R}_{j_\ell}\right)$  for any $n\geq 1$.

We write $\mathcal R= \mathcal R(1)$ for short  and notice that $\displaystyle \mathcal R^{(\ell)}  : = \underbrace{\mathcal R\circ \ldots \circ \mathcal R}_{ \ell \   times}= \sum_{n \geq 1} \left( \sum_{j_1+ \ldots + j_\ell=n}{\mathcal R}_{j_1}\ldots {\mathcal R}_{j_\ell}\right)$; it is thus natural to set $\displaystyle \mathcal R_n^{(\ell)}:= \sum_{j_1+ \ldots + j_\ell=n}{\mathcal R}_{j_1}\ldots {\mathcal R}_{j_\ell}$ so that $\displaystyle \mathcal R^{(\ell)}= \sum_{n \geq 1} \mathcal R_n^{(\ell)}$ for any $\ell \geq 1$. As usual, we set $\mathcal R_n^{(0)}=I$.

 \begin{theo} \label{theogouezel}\cite{gouezel}
 Let $(\mathcal R_n)_{n \geq 1}$  be an aperiodic renewal sequence of operators acting on the  Banach space $({\mathcal B}, \vert \cdot \vert_{\mathcal B })$ and    assume that 
 there exist  $C>0,\ \beta \in ]0, 1[$ and   a slowly varying function $L$ such that

 R4($ L,\beta)$-$    \quad \displaystyle     \Vert \mathcal R_n\Vert_{{\mathcal B }} \leq C \frac{L(n)}{n^{1+\beta}}$;

 R5($ L,\beta)$-$ \quad \displaystyle      \sum_{j>n}r_j \thicksim \frac{L(n)}{n^\beta}$ as $n \to +\infty$.

\noindent Then the sequence $(n^{1-\beta} L(n) {\mathcal T}_n)_{n\geq 1}$ converges in $(\mathcal L({{\mathcal B }}), \Vert \cdot \Vert_{\mathcal B })$ to the operator $d_\beta \Pi_\mathcal R$, with $d_\beta = \frac{1}{\pi}\sin (\beta \pi)$. 

 \end{theo}

In order to control the behavior as $n \to +\infty$ of  each term $\mathcal R_n^{(\ell)}$ for  $\ell \geq 1  $, we reinforce the hypotheses and obtain the following statement.

 \begin{theo} \label{theogouezel+}
 Let $(\mathcal R_n)_{n \geq 1}$  be an aperiodic renewal sequence of operators acting on the  Banach space $({\mathcal B}, \vert \cdot \vert_{\mathcal B })$. Assume that 
 there exist  $\beta \in ]0, 1[$ and  a slowly varying function $L$ such that    the sequence $\displaystyle \left({n^{1+\beta}\over L(n)} \mathcal R_n\right)_{n \geq 1}$ converges  in $(\mathcal L({{\mathcal B }}), \Vert \cdot \Vert_{\mathcal B })$  to some
  bounded operator $\mathcal E$. 
Then  
\begin{enumerate}
 \item there exists a strictly positive constant $C $ such that  for any $n, \ell \geq 1$, 
\begin{equation} \label{evgj}
\Vert \mathcal R_n^{(\ell)}\Vert_{  \mathcal B} \leq  C \ell^2    {L(n) \over n^{1+\beta}};
\end{equation}
\item for any $\ell \geq 1$, the sequence $\displaystyle \left({n^{1+\beta}\over   L(n)} \mathcal R_n^{(\ell)}\right)_{n \geq 1}$ converges    in $(\mathcal L({{\mathcal B }}), \Vert \cdot \Vert_{\mathcal B })$ to the operator $\displaystyle \mathcal E_\ell:= \sum_{i=0}^{\ell-1} \mathcal R^{(i)} \mathcal E\mathcal R^{(\ell-i-1)}$.
\end{enumerate}
 \end{theo} 
\begin{proof}  
 1. Let $(\alpha_\ell)_{\ell \geq 1}$ be  a sequence with $ \ell L(\alpha_\ell) \sim \alpha_\ell^{\beta}$. By \cite[Proposition 1.5, Theorem 1.6]{gouezel}, there exists a strictly positive constant$C$ such that
 $$
 \Vert \mathcal R_n^{(\ell)}\Vert_{  \mathcal B} \leq  \left\{
\begin{array}{lll}
  {C/ \alpha_\ell}   &  {\rm if} &  1\leq n \leq \alpha_\ell,
  \\
  \ & \ & \ 
  \\
  C \ell{L(n)\over n^{1+\beta}}&  {\rm if}   &  n \geq \alpha_\ell.
 \end{array}
 \right.
 $$
 Notice that, for $\ell$ large enough    and $1\leq n \leq \alpha_\ell$,
 $$
 {1\over \alpha_\ell}\leq 2  \ell 
 {L(\alpha_\ell)\over \alpha_\ell^{1+\beta}}
 \leq 2 \ell \ {L(n)\over n^{1+\beta}} \  {L(\alpha_\ell) \over L(n)} \leq C \ell^2{L(n)\over n^{1+\beta}},
 $$  
where the last inequality follows by   Potter's lemma.

 \noindent 2. We proceed by induction.  By hypothesis,  the assertion holds for $\ell=1$.  
 Assume that   
  $\displaystyle \left({n^{1+\beta}\over L(n)} \mathcal R_n^{(\ell)}\right)_{n \geq 1}$ converges  in $(\mathcal L({{\mathcal B }}), \Vert \cdot \Vert_{\mathcal B })$ to the operator $\displaystyle \mathcal E_\ell:= \sum_{i=0}^{\ell-1} \mathcal R^{(i)} \mathcal E\mathcal R^{(\ell-i-1)}$ for some $\ell \geq 1$. We fix $1\leq m \leq n/2$  and decompose $\displaystyle  \mathcal R_n^{(\ell+1)}= \sum_{i=1}^{n-1} \mathcal R_i^{(\ell)}  \mathcal R_{n-i} $ as 
\begin{align*}
   \mathcal R_n^{(\ell+1)}  = A_n(m)+ B_n(m) +C_n(m)
   \end{align*}
with $\displaystyle A_n(m):= 
     \sum_{i=1}^m  \mathcal R_i^{(\ell)}  \mathcal R_{n-i}, \ B_n(m)= \sum_{i=m+1}^{n- m-1}  \mathcal R_i^{(\ell)}  \mathcal R_{n-i}$
     and $\displaystyle  C_n(m)=   \sum_{i=n-m}^{n-1}   \mathcal R_i^{(\ell)}  \mathcal R_{n-i}$.
     
 (i) On the one hand, the  sequences  $\left({n^{1+\beta}\over L(n)}A_n(m)\right)_{n\geq 1}$ and $\left({n^{1+\beta}\over L(n)} C_n(m)\right)_{n\geq 1}$  converge   in $(\mathcal L({{\mathcal B }}), \Vert \cdot \Vert_{\mathcal B })$ to 
$\displaystyle \left(\sum_{i=1}^m\mathcal R_i^{(\ell)}\right) \mathcal E$ and  $\displaystyle \mathcal E_\ell \sum_{i=1}^m \mathcal R_i$, respectively. 

 (ii) On the other hand, 
by combining  {\it R4}$(L,\beta)$ and 
(\ref{evgj}), there exists  a strictly positive constant $C$ such that
$$
\displaystyle \bigg\Vert \sum_{i=m+1}^{\lfloor  n/2\rfloor}   \mathcal R_i^{(\ell)}    \mathcal R_{n-i}\bigg\Vert_{  \mathcal B} \leq C \ell^2 \  {L(n) \over n^{1+\beta}} \sum_{i=m+1}^{\lfloor  n/2\rfloor}   { L(i)  \over i^{1+\beta}}
\leq C \ell^2 \  {L(n) \over n^{1+\beta}} \sum_{i=m+1}^{+\infty}   { L(i)  \over i^{1+\beta}}.
$$
By symmetry, the same inequality holds for   $\displaystyle  \bigg\Vert \sum_{i=\lfloor  n/2\rfloor+1}^{n-m-1}   \mathcal R_i^{(\ell)}    \mathcal R_{n-i}\bigg\Vert_{  \mathcal B}.$ Hence, uniformly in $n\geq 2m$, 
 the sequence $\displaystyle \left({n^{1+\beta}\over L(n)} B_n(m) \right)_m$ converges to $0$ as $m \to +\infty$.

 By letting first $n \to +\infty$ and then $m \to +\infty$,   we obtain that the sequence $\left({n^{1+\beta}\over L(n)}\mathcal R_n^{(\ell+1)}\right)_n$ 
 converges to  
$\displaystyle \left(\sum_{i\geq 1}\mathcal R_i^{(\ell)}\right) \mathcal E+ \mathcal E_\ell \sum_{i\geq 1}  \mathcal R_i = \mathcal R^{(\ell)}\mathcal E+ \mathcal E_\ell \mathcal R= \mathcal E_{\ell+1}$.
\end{proof}


\section{Proof of Theorem \ref{theo1}: the recurrent  cases }\label{sectionrecurrent}

 In this section, we focus on the null recurrent cases $(\bf Z,Z)$ and $\bf (P,Z)$; in particular,  we first control the spectrum of the operator $Q$ and then investigate the  behavior of the associated sequence $(Q_n)_{n \geq 0}$ to be defined.

\subsection{ A toy model in the case \texorpdfstring{$(\bf Z, Z)$}{ZZ}}\label{toy}
We consider here the case when $D=1$ and $D'=-1$. This hypothesis implies that when the process $(X_n)_{n \geq 0}$ leaves one of the half lines $\mathbb Z^{*+}$ or $\mathbb Z^{*-}$  it reaches the site $0$. The condition $\mathbb E[\xi_1]=\mathbb E[\xi'_1]=0$ implies that $\mathbb P_x[C_1<+\infty]=1$  so that the transition matrix $ Q = (Q(x, y))_{x, y \in \mathbb Z}$ of the switching  subprocess  is given by 
\begin{equation*} 
Q(x, y)  =   
 \left\{
   \begin{array}{clcll}
  {\bf 1}_{\{0\}}(y)&  if &x \neq 0,& \  &  \   \\
    \displaystyle {\mu_0(y)\over 1-\mu_0(0)}  &  if &x = 0& and &y \neq 0.   
   \end{array}
    \right.
\end{equation*}
Consequently, the excursions between two successive switching times are independent, which simplifies the analysis.  Let us detail for instance the way to estimate the asymptotic behavior of  $\mathbb P_0[X_n=y]$, the other cases may be treated in a similar way using the decomposition of trajectories presented in subsection \ref{decomposition}.

Let $({\bf t}_\ell)_{\ell \geq 0}$ be the successive visit times of the origin defined by ${\bf t}_0=0$ and 
${\bf t}_\ell=\inf \{n>{\bf t}_{\ell-1}\mid X_n=0\}$ for $\ell\geq 1$. For $\ell \geq 1$, the random variables ${\bf t}_{\ell}-{\bf t}_{\ell-1} $ are iid  so that  $({\bf t}_\ell)_{\ell \geq 0}$ is a classical random walk on $\mathbb N$ with jump distribution $\mu_{\bf t}$ given by:  for any $n \geq 1$, 
\begin{equation*}
\mu_{\bf t}(n) =  
 \left\{
   \begin{array}{clc}
\mu_0(0)&  if &n = 1,\\
\displaystyle \sum_{y \neq  0} \mu_0(y) \mathbb P_y[ C_1=n-1]&  if &n \geq 2.\\  
   \end{array}
    \right.
\end{equation*}
Consequently, by Lemma \ref{MAJ+asym}, as $n \to +\infty$, 
\begin{align*}
\mu_{\bf t}(n)&=  \sum_{y \leq  -1} \mu_0(y)\  \mathbb P[ \tau^{S}(y)=n-1]+\sum_{y \geq  1} \mu_0(y)\  \mathbb P[ \tau^{S'}(y)=n-1] 
 \\
& \sim
\left(c \sum_{y \leq  -1} \mu_0(y) V_{*+}(\vert y\vert)+ c'\sum_{y \geq  1} \mu_0(y)\  V'_{*-}(y) 
\right) /n^{3/2}.\end{align*}  
By using \cite[Theorem B]{doney1997}, one gets  for any $n\geq 1$,
\begin{align*}
\mathbb P_0[X_n=0]&=\sum_{\ell=0}^{+\infty} \mathbb P_0[{\bf t}_\ell= n] 
\\
&\sim 
2 \left(c \sum_{y \leq  -1} \mu_0(y) V_{*+}(\vert y\vert)+ c'\sum_{y \geq  1} \mu_0(y) V'_{*-}(y)
\right) /\sqrt{n}.
\end{align*}
Similarly, if $y \neq 0$ (say for instance $y\leq -1$), one gets
\begin{align*}
\mathbb P_0[X_n=y]&=
 \sum_{k=0}^{n-1} \sum_{\ell=0}^{+\infty} \mathbb P_0[X_n= y, {\bf t}_\ell =k,   {\bf t}_{\ell+1}>n] 
\\
&=
\sum_{k=0}^{n-1} \underbrace{\left( \sum_{\ell=0}^{+\infty} \mathbb P_0[{\bf t}_\ell= k]\right)}_{ \sim c_1/\sqrt{k}} \underbrace{\mathbb P[ \tau^S(0) > n-k, S_{n-k}=y]}_{ \sim c_2(y)/(n-k)^{3/2}}
\\
& \sim c_3(y)/\sqrt{n}
\end{align*}
 for some strictly positive constants  $c_1, c_2(y)$ and $c_3(y)$ that can be computed explicitly.  
 
 For general oscillating random walks, the   change of regime between the three media  $\mathbb Z^{*-},
 \{0\}$ and $\mathbb Z^{*+}$ is taken into account by using the sequence of switching times  $(C_\ell)_{\ell \geq 0}$.  This sequence corresponds more or less to the sequence $({\bf t}_\ell)_{\ell \geq 0}$ when $DD'=-1$.
Unfortunately, when $DD'\neq -1$, the random variables $C_\ell-C_{\ell-1}, \ell \geq 1,$ are no longer independent and the above argument  fails.  Alternatively, we use an extension of Doney's theorem due by S. Gou\"ezel \cite{gouezel}, whose implementation is much more subtle.

{\bf From now on, we assume that $DD'\neq -1$.}

\subsection{On the  switching sequence \texorpdfstring{$(Q_n)_{n \geq 1}$}{Q} in the  cases   \texorpdfstring{${\bf (P, Z)}$}{PZ} and   \texorpdfstring{${\bf (Z, Z)}$}{ZZ}   }

  By (\ref{ switchingtime}), it is clear that $C_{\ell+1}= \tau^S(X_{C_\ell})$ when  $X_{C_\ell}\leq -1$ and $C_{\ell+1}=  \tau^{S'} (X_{C_\ell})$ when  $X_{C_\ell}\geq  1$. Consequently,  the  behavior as $n \to +\infty$ of  the quantities $ Q_n^{(\ell)}(x, y):= \mathbb{P}_x[C_\ell=n, X_n=y]$ for $\ell\geq 1$    is  closely related to the distributions of  $\tau^S$ and $\tau^{S'}$ when $x \neq 0$; in particular, by Lemma \ref{MAJ+asym}, its dependence on   $y$ is  expressed in terms of the  weakly and strictly renewal functions associated with $S$ and $S'$. This yields us  to choose a  suitable  Banach space on which the action of $Q$ has ``nice'' spectral properties - as compacity or quasi-compacity - and also    contains these renewal functions. 
  The fact that these functions   are   sublinear leads us  to examine the action of $Q$ on  the space $\mathcal B_{\Psi}$   of complex valued functions on $\mathbb Z$ defined by 
 $$  \mathcal B_{\Psi} :=\left\{f:\mathbb{Z} \to \mathbb{C}: \vert f \vert_{\mathcal B_\Psi}:=
  \sup_{x \in \mathbb Z}\frac{|f(x)|}{ \Psi(x)}  <+\infty 
 \right\},
 $$
 where $\Psi: \mathbb Z\to \mathbb R^+$ is a function such that $\displaystyle \lim_{x\to\pm \infty}{\Psi(x)\over 1+\vert x\vert}=+\infty$. 
 This condition ensures that the renewal functions $ V_+, V_{*+}, V_-, V_{*-}$  associated to  the random walk $S$, and the corresponding ones associated to   $S'$,  belong to $\mathcal B_\Psi$ and that  the    inclusion map $i: \mathbb L^\infty(\mathbb{Z}) \hookrightarrow {{\mathcal B_\Psi}}$ is compact which is a key point  in the sequel. 
 
 In the cases   $\bf (Z, Z)$, $\bf (P, Z)$ and $\bf (Z, P)$,  we only have polynomial moment assumptions and we fix the function $\Psi(x):= 1+\vert x\vert ^{1+\delta}$ with $\delta >0$. Let us explain briefly this choice.

By Lemma \ref{MAJ+asym}, the renewal functions   of $S$ and $S'$  belong to ${\mathcal B_\Psi}$ for any $\delta\geq 0$.  The same property holds for   the functions $  {\bf V}_{n,y}$ for $n \geq 1$ and $ y \in \mathbb Z$  defined  in \eqref{hny} and  \eqref{h'ny} ; furthermore,  for any $y \in \mathbb Z$, the sequence     $(n^{3/2}  {\bf V}_{n, y})_{n \geq 1}$    is bounded    in ${\mathcal B_\Psi}$. The condition $\delta>0$ is required to obtain the convergence of  this sequence    in ${\mathcal B_\Psi}$; its limit is $  {\bf V}_y$ defined  by $  {\bf V}_0=0$ and 

$\bullet$ for $y\leq -1$, 
\begin{equation}\label{hy}
  {\bf V}_{ y}: x \mapsto \left\{\begin{array}{cll}
\frac{1}{\sigma \sqrt{2\pi}  }  V_{*+}(\vert x\vert)V_-(\vert y\vert)
 & if & x \leq -1\\
 0&  if & x \geq 0
\end{array} 
\right.  \end{equation}

$\bullet$ for $y\geq  1$, 
\begin{equation}\label{h'y}
  {\bf V}_{ y}: x \mapsto \left\{\begin{array}{cll}
 0&  if & x \leq 0 
 \\
 \frac{1}{\sigma' \sqrt{2\pi}  }  V'_{*-}(  x )V'_+( y )
 & if & x \geq  1 
\end{array}
\right. .\end{equation}
For this reason, we assume from now on  $\delta>0$.   We also suppose in the sequel  that $\delta<\delta_0$, where  $\delta_0$ is given by the moment conditions.

 The  case  $\bf (Z, Z)$  was  well studied in \cite{peignevo}, where it was proved that  the operator  $Q$  acts on ${\mathcal B_\Psi}$ as a markovian compact operator. In the following proposition, we make precise this statement  and extend it to the case ${\bf (P, Z)}$.

\begin{prop} \label{QZZ}  \underline{Cases ${\bf (Z, Z)}$ and ${\bf (P, Z)}$} 

Assume hypotheses {\bf O1--O4}  and moment conditions $\bf (Z, Z)$ and $\bf (P, Z)$. Set   $\Psi(x)= 1+\vert x\vert ^\delta$ where $0<\delta <\delta_0$. Then the following statements are true.
\begin{enumerate}[(i) ]
    \item the map $Q$ acts on ${{\mathcal B_\Psi}}$ and $Q({{\mathcal B_\Psi}})$ is included in the space  $\mathbb L^\infty(\mathbb{Z})$ of  bounded functions on $\mathbb Z$ endowed with the norm $\vert \cdot \vert_\infty$.
    \item  $Q$ is a compact operator on ${{\mathcal B_\Psi}}$ with spectral radius $\rho_{\mathcal B_\Psi}  =1$ and with the unique and simple dominant eigenvalue $1$ (the rest of  the spectrum of $Q$   on ${\mathcal B_\Psi}$ is contained in a disk of radius strictly less than $1$);
      \item the spectral radius $\rho_{\mathcal B_\Psi}(z)$ of $Q(z):= \sum_{n \geq 1} Q_n z^n$ is strictly less than $1$ for $z \in \overline{\mathbb{D}}\setminus\{1\}$.  
\end{enumerate}
 Consequently, the operator $Q$ on ${\mathcal B_\Psi}$ may be decomposed as
 $
Q= \Pi_Q+ Q'
 $
where 
\\
\indent $\bullet$ \ \ $\Pi_Q$ is the eigenprojector from ${\mathcal B_\Psi}$ to $\mathbb C \bf{1}$ corresponding to the eigenvalue $1$ and $\Pi_Q(\varphi)= \nu_Q(\varphi) \mathbf{1}$, where $\nu_Q$ is the unique $Q$-invariant probability measure on $\mathbb Z$;

\indent $\bullet$ \ \ the spectral radius of $Q'$  on ${\mathcal B_\Psi}$ is strictly less than $ 1$; 

 $\bullet$ \ \ $\Pi_Q Q'= Q'\Pi_Q = 0$.
 
\noindent  Furthermore, for each $n\geq 1,$ 
 the following properties of the sequence $(Q_n)_{n\geq 1}$ hold.
\begin{enumerate}[(a)]
\item The operator $Q_n$ acts on $\mathcal B_\Psi$ and  its induced norm $ \Vert Q_n \Vert_{\mathcal B_\Psi}$ in $\mathcal L(\mathcal B_\Psi)$ satisfies 
 $\displaystyle \Vert Q_n \Vert_{\mathcal B_\Psi} = O\left(\dfrac{1}{n^{3/2}}\right).$ 
      \item For any $n\geq 1$, it holds that  $\Pi_Q Q_n \Pi_Q= r_n \Pi_Q $ where
$\displaystyle 
    r_n:=\nu_Q(Q_n {\bf 1})=\displaystyle \sum_{x\in \mathbb{Z}}\nu_Q(x)\mathbb{P}_x[\mathcal \mathcal  C_1=n]\geq 0. 
$  Furthermore,   
$\displaystyle 
\sum_{j>n}r_j   \thicksim 
\left\{
\begin{array}{cl}
\displaystyle \frac{ 2(c\nu_Q(\check{V}_{*+})+c'\nu_Q(V'_{*-})) }{\sqrt{n}} & in \ the \ case \ {\bf(Z, Z)},
\\ 
\displaystyle \frac{2c'\nu_Q(V'_{*-})}{\sqrt{n}}  & in \ the \ case \ {\bf(P, Z)},
\end{array}
\right.
$
 where $\check{V}_{*+}$ denotes the function $x \mapsto V_{*+}(-x)$ for any $x \leq 0$. 
\item The sequence $(n^{3/2} Q_n)_{n \geq 1}$ converges in $ (\mathcal  L( \mathcal B_\Psi), \Vert . \Vert_{\mathcal{B}_\Psi}) $  to the operator $\mathcal E$  whose transition matrix is given by: for any $x, y \in \mathbb Z$,

$\bullet$ in    the   case   {\bf(Z, Z)},
\[
  \mathcal E(x, y) = 
\left\{
\begin{array}{clc}
  \displaystyle \frac{1}{\sigma\sqrt{2\pi}} h _{*+}(|x|)   \sum_{w \geq 1} V_-(w) \mu( w+y)
 & if\ \bigl( x\leq   -1\ and \ y\geq 0\bigr),
\\
\displaystyle 0
 & if\  \bigl(x= 0  \ and \ \ y\neq 0\bigr),
\\
    \displaystyle\frac{1}{\sigma'\sqrt{2\pi}} V'_{*-}(x)  \sum_{w \geq 1} V'_+(w) \mu'( -w+y) 
 & if\  \bigl(x\geq  1 \ and \ \ y\leq 0\bigr). 
\end{array}
\right.
\]
$\bullet$ in    the   case   {\bf(P, Z)},
\[
  \mathcal E(x, y) = 
\left\{
\begin{array}{cll}
  0  & if\  \bigl(x\leq -1\ and \  y\geq 0\bigr) \   or  \bigl(x=0\ and \ y \neq 0\bigr),
  \\
   \displaystyle \frac{1}{\sigma'\sqrt{2\pi}} V'_{*-}(x)  \sum_{w \geq 1} V'_+(w) \mu'( -w-y)
 & if\  \bigl(x\geq  1 \ and \ y\leq 0\bigr). 
\end{array}
\right.
\] 
\end{enumerate}

\end{prop}
Assertion $(c)$ is not used in the proof of Theorem \ref{theo1} but it is useful in sections \ref{sectiontransientZP} and \ref {sectiontransientNPPP}. 
\begin{proof}  For the sake of completeness, we  detail  here the proof  developed in  \cite[Proposition 3.5, Lemma 3.7 and Proposition 3.8]{peignevo} in the case $({\bf Z, Z})$ and explain how to extend it  to the case  $({\bf P, Z})$.   

 \begin{enumerate}[$(i)$]
     \item  
  By formula \eqref{Qswitchingtime-finite2}, for any $\varphi \in {{\mathcal B_\Psi}}$ and $x \leq -1$, it holds that   
  \begin{align}\label{bound} 
         |Q\varphi(x)| \leq \ \sum_{ y\geq 0}   \sum_{t = 0}^{\vert x\vert-1 } \mu_{*+}(\vert x\vert+y-t)\, \vert  \varphi(y)|
         &\leq  \vert  \varphi \vert_{\mathcal B_\Psi}  \sum_{ y\geq 0}(1+ y^{1+\delta})  \, \mu_{*+}[ y, +\infty )\notag\\ 
         &\leq \vert  \varphi \vert_{\mathcal B_\Psi} \left(\mathbb{E}[S_{\tau^S_{*+}} ] + \mathbb{E}[ (S_{\tau^S_{*+}}) ^{2+\delta}]\right).
     \end{align}
     The same inequality holds  when $x\geq 1$, replacing $ S_{\tau^S_{*+}}$ by  $S'_{\tau^{S'}_{*-}}$.
     \footnote{Notice that, by using Lemma  \ref{MAJ+asym}, we   obtain
     $\displaystyle  |Q\varphi(x)| \preceq \vert  \varphi \vert_{\mathcal B_\Psi} \vert x\vert \mathbb E[(\xi_1^+)^{3+\delta}]$ which is not sufficient to establish that $Q$ acts continuously from $\mathcal B_\Psi$ to $\mathbb L^\infty(\mathbb Z)$; this property   is of major importance  to prove that $Q$ has bounded powers on $\mathcal B_\Psi$ (see part $(ii)$ of the present proof).}
     
     At last $\displaystyle  |Q\varphi(0)| \leq  \vert  \varphi \vert_{\mathcal B_\Psi}  {\mathbb E[1+ \vert \eta_1\vert^{1+\delta}]\over 1-\mu_0(0)}$. 
By  \cite[Theorem 1]{ChowLai}, the quantities $\mathbb{E}[(S_{\tau^S_{*+}})^{2+\delta}]$  and $\mathbb{E}[\vert S'_{\tau^{S'}_{*-}} \vert^{2+\delta}]$ are both finite   in the case $({\bf Z, Z})$  since $\mathbb{E}[(\xi ^+_n)^{3+\delta}],\mathbb{E}[(\xi'^-_n)^{3+\delta}]<+\infty$;    in the   case $({\bf P, Z})$,   this property also  holds  for $\mathbb{E}[(S_{\tau^S_{*+}})^{2+\delta}]$   by \cite[Theorem 3.1]{gut}.        
    
     \item By \eqref{bound}, the operator $Q$ is bounded from  ${{\mathcal B_\Psi}}$ into $\mathbb L^\infty(\mathbb{Z})$; as the inclusion map $i: \mathbb L^\infty(\mathbb{Z}) \hookrightarrow {{\mathcal B_\Psi}}$ is compact, the operator $Q$ is also compact on ${{\mathcal B_\Psi}}$.

        Let us compute the spectral radius $ \rho_{\mathcal B_\Psi}$ of $Q$. The fact that $Q$ is a stochastic matrix yields $ \rho_{\mathcal B_\Psi} \geq 1$. To prove $ \rho_{\mathcal B_\Psi}\leq 1$,  it suffices to show that the sequence of iterates $(Q^n)_{n \geq 0}$ is bounded in  ${{\mathcal B_\Psi}}$. For any $n \geq 1, \varphi \in \mathcal B_\Psi$ and $x \in \mathbb{Z}$,
    \begin{align*}
        \vert Q^{n}\varphi(x) \vert \leq \displaystyle \sum_{y\in \mathbb{Z}} Q^{n-1}(x,y) \vert Q\varphi(y) \vert \leq \vert Q\varphi \vert_{\infty} \displaystyle \sum_{y\in \mathbb{Z}} Q^{n-1}(x,y) =\vert Q\varphi \vert_{\infty}.
    \end{align*}
Together with \eqref{bound}, it implies
    \begin{align*}
        \vert Q^{n} \varphi \vert_{\mathcal B_\Psi}  \leq \vert Q^n \varphi \vert _\infty \leq |Q\varphi|_\infty \leq   \vert  \varphi \vert_{\mathcal B_\Psi}\left(2 \left(\mathbb E[(S _{\tau^S_{*+}})^{2+\delta}] +
         \mathbb E[\vert S'_{\tau^{S'}_{*-}}\vert ^{2+\delta}]  \right)+{\mathbb E[1+ \vert \eta_1\vert^{1+\delta}]\over 1-\mu_0(0)}\right) .
    \end{align*}
Hence $\Vert Q^n\Vert_{ \mathcal B_\Psi} \leq 2  \left(\mathbb E[(S _{\tau^S_{*+}})^{2+\delta}] +
         \mathbb E[\vert S'_{\tau^{S'}_{*-}}\vert ^{2+\delta}] \right)+{\mathbb E[1+ \vert \eta_1\vert^{1+\delta}]\over 1-\mu_0(0)} $ for any $n \geq 1$ and so 
 $ \rho_{\mathcal B_\Psi} \leq 1.$  Finally $\rho_{\mathcal B_\Psi}=1$.
 
Let us control the peripheral spectrum of $Q$. Let $\theta \in \mathbb{R}$ and $\psi \in {{\mathcal B_\Psi}}$ such that $Q\psi = e^{i\theta}\psi$. Since $Q(\mathcal B_\Psi) \subset \mathbb L^\infty(\mathbb{Z})$, the function $\psi$ is bounded and $\vert \psi \vert \leq Q \vert \psi \vert$.  Consequently,  $\vert \psi \vert_{\infty} -\vert \psi \vert$ is  non negative and  superharmonic    (ie  $Q(\vert \psi \vert_{\infty} -\vert \psi \vert) \leq \vert \psi \vert_{\infty} -\vert \psi \vert$)  on the unique essential class  $\mathcal{I}_{\bf C}$ of $\mathcal X$,  which is recurrent since $\mathbb E[\xi_1]\geq 0$ and $\mathbb E[\xi'_1]=0$. By the classical denumerable Markov chains theory, it is thus constant on $\mathcal{I}_{\bf C}$ which implies that $\vert \psi\vert$ is constant on $\mathcal{I}_{\bf C}$.

Without loss of generality, we may assume that $\vert \psi(x) \vert =1$ for any $x \in \mathcal{I}_{\bf C}$,  ie   $\psi(x) = e^{i\phi(x)}$ for some $\phi(x) \in \mathbb{R}$. The equation $Q\psi = e^{i\theta}\psi$ can be written as 
    \begin{align*}
 \forall x \in \mathcal{I}_{\bf C}, \quad  \sum_{y \in {\mathcal{I}_{\bf C}}} Q(x,y) e^{i(\phi(y) - \phi(x))} = e^{i\theta}. 
    \end{align*}
 By convexity, one readily gets $e^{i\theta} =e^{i(\phi(y) - \phi(x))}$ for all points $x, y \in \mathcal I_{\bf C}$ such that $Q(x, y)>0$. We fix two points $x_1, x_2  \in \mathcal{I}_{\bf C}\setminus\{0\}$. Since  $Q(x_1, 0), Q(x_2, 0) >0$, it holds $e^{i \phi(x_1)} = e^{i (\phi(0)-\theta)}= e^{i \phi(x_2)}$, hence $e^{i\phi} $ is constant on $\mathcal{I}_{\bf C}\setminus\{0\}$.  By hypothesis {\bf O3}, we obtain  $D\geq 2$   or $D'\leq -2$. Assume for instance $D\geq 2$ and  fix $y_0 \in [1, \ldots, D[$;  the points $-1$ and $y_0$ belong to $\mathcal I_{\bf C}\setminus\{0\}$, hence $e^{i \theta} =  e^{i(\phi(y_0)-\phi(-1))} =1$. Therefore,  the function $\psi$ is harmonic on $\mathcal{I}_{\bf C}$ and  by Liouville's theorem, it is constant on this set. 
   
For any $x \in \mathbb Z$, since   $\left(Q(x,y)>0 \Longrightarrow y \in \mathcal{I}_{\bf C}\right)$,   we obtain  
 $
       \psi(x)=Q\psi(x)= \displaystyle \sum_{y \in {\mathcal{I}_{\bf C}}} Q(x,y)\psi(y)= \psi(y_0), 
 $
therefore  $\psi$ is constant on $\mathbb{Z}$. 
 \item 
 The argument  is close to the one above. For any $z\in \overline{\mathbb{D}}\setminus\{1\}$, the operator $Q(z)$ is compact on ${{\mathcal B_\Psi}}$ with spectral radius $ \rho_{\mathcal B_\Psi}(z)\leq 1$. We prove that $ \rho_{\mathcal B_\Psi}(z) \neq 1$ by contraposition.  Suppose $\rho_{{{\mathcal B_\Psi}}}(z)=1$; in other words, there exist $\theta \in \mathbb{R}$ and $\psi \in {{\mathcal B_\Psi}}$ such that $Q(z)\psi= e^{i\theta} \psi$. Since $Q$ is bounded from ${{\mathcal B_\Psi}}$ into $\mathbb L^\infty(\mathbb{Z})$ and $0\leq \vert  \psi| \leq Q\vert  \psi|$. As above, we conclude that  $\vert  \psi|$ is  constant on  $\mathcal{I}_{\bf C}$; we may assume  
 $\vert  \psi(x)|=1$ for any  $x \in \mathcal{I}_{\bf C}$, ie $\psi(x)=e^{i\phi(x)}$  for  some function  $\phi:\mathcal{I}_{\bf C} \to \mathbb R$. For any $x \in \mathcal{I}_{\bf C}$, we get
     \begin{align*}
         Q(z)\psi(x)= e^{i\theta}\psi(x) &\Longleftrightarrow   \sum_{n\geq 1}   \sum_{y \in \mathcal{I}_{\bf C}} z^n e^{i(\phi(y)-\phi(x))} \mathbb{P}_x[C_1=n; X_n=y]=e^{i\theta} 
     \end{align*}
with $ \displaystyle \sum_{n\geq 1}   \sum_{y \in \mathcal{I}_{\bf C}} \mathbb{P}_x[C_1=n; X_n=y]=1$.\\
 By convexity, we obtain  $z^n e^{i(\phi(y)-\phi(x))}=e^{i\theta}$ for all $x, y \in \mathcal{I}_{\bf C}$ and $n \geq 1$ such that $Q_n(x, y)>0$. Fix $x_1 \in \mathcal I_{\bf C}$ such that $x_1 \leq -1$;  since  the measure $\mu$ is  aperiodic, there exists  $n \geq 1$ such that $Q_n(x_1, 0)  Q_{n+1}(x_1, 0) >0$ so that  $z^n e^{i(\phi(0)-\phi(x_1))}=z^{n+1} e^{i(\phi(0)-\phi(x_1))}=e^{i\theta}$. Consequently 
  $z=1$. Contradiction. 
  \end{enumerate}

Let us  focus  on the sequence $(Q_n)_{n \geq 1}$.
  \begin{enumerate}[(a)]
  \item    When $x \leq -1$, one gets
  
   $\bullet$ in the case $\bf (Z, Z)$,  by using Lemma \ref{MAJ+asym}$(d)$,
\begin{align*}
 \vert Q_n\varphi(x) \vert   
      &=   \sum_{y \geq 0}\mathbb{P}[\tau^S(x)=n, x+S_n=y]\, |\varphi(y)|\\
      &\preceq {1+\vert x\vert\over n^{3/2}}\left(\sum_{y \geq 0} \sum_{z > y} z \mu(z) \right)  |\varphi(y)|\\
      & \preceq \frac{1+\vert x\vert}{n^{3/2}}\vert \varphi\vert_{\mathcal B_\Psi} \underbrace{\sum_{y \geq 0} (1+y^{1+\delta})\left(\sum_{z > y} z \mu(z)\right)}_{\preceq  \  \sum_{z \geq 1}z^{3+\delta}\mu(z)\ =\  \mathbb E[(\xi_1^+)^{3+\delta}]}.
  \end{align*}
  
    $\bullet$ in the case $\bf (P, Z)$,  since  $\delta \in ]0, \delta_0[$, by \eqref{tsejvhtb},  
          \begin{align*}
 \vert Q_n\varphi(x) \vert  &\leq  \sum_{y \geq  0} \mathbb{P}_{x}[ C_1=n; X_n=y]\, |\varphi(y)|\\ 
      &=   \sum_{y \geq 0}\mathbb{P}[\tau^S(x)=n, x+S_n=y]\, |\varphi(y)|\\
      &\leq C_\delta \frac{(1+\vert x\vert)^{1+\delta}}{ n^{\frac{3}{2}+\delta}} \vert \varphi\vert_{\mathcal B_\Psi}
      \underbrace{\sum_{y \geq 0} (1+\vert y\vert^{1+\delta})\mu(y, +\infty)}_{\preceq  \ \mathbb E[ (\xi_1^+)^{2+\delta}]\ <+\infty}.        \end{align*}

Similar inequalities hold   for both cases $\bf (Z, Z)$ and $\bf (P, Z)$ when $x\geq 0$.
This completes the proof.
\item For any $\varphi \in {{\mathcal B_\Psi}}$ and $n\geq 1$,
 $$
         \Pi_Q Q_n \Pi_Q (\varphi)= \nu_Q(\varphi)\Pi_Q(Q_n \mathbf{1})  = \nu_Q(\varphi)    \left(\sum_{x\in \mathbb{Z}} \nu_Q(x) \mathbb{P}_x[C_1=n]\right){\bf 1}  = r_n\Pi_Q(\varphi) 
 $$
     with $r_n=  \displaystyle  \sum_{x\in \mathbb{Z}} \nu_Q(x) \mathbb{P}_x[C_1=n]\geq 0$.    Consequently, 
     
$\bullet$ in the case $\bf (P, Z)$,  
\begin{align*}
         \displaystyle \sum_{j>n} r_j &= \displaystyle \sum_{j>n}  \displaystyle \sum_{x\in \mathbb{Z}} \nu_Q(x) \mathbb{P}_x[C_1=j]\\
         &= \displaystyle \sum_{x\in \mathbb{Z}} \nu_Q(x) \mathbb{P}_x[C_1>n]\\
         &=\displaystyle \sum_{ x \leq -1 } \nu_Q(x) \mathbb{P}[\tau^S(x)>n]+  \sum_{x \geq  1} \nu_Q(x) \mathbb{P}[\tau^{S'}(x)>n]+ \nu_Q(0) \mu_0(0)^n\\
         &\sim \frac{2c'}{\sqrt{n}} \displaystyle \sum_{x \geq 1}\nu_Q(x) V'_{*-}(x) \quad {\rm by} \ \eqref{jvzhgqf} \ {\rm and} \ \eqref{tauS>n}\\
         &=\frac{2c' \ \nu_Q(V'_{*-})}{\sqrt{n}};
     \end{align*} 
   We use here the fact that $0<\nu_Q(V'_{*-})<+\infty$. Indeed, the eigenprojector $\Pi_Q$ acts on ${\mathcal B_\Psi}$ and $ V'_{*-} \in {\mathcal B_\Psi}$, hence    $\nu_Q(V'_{*-}) <+\infty$. On the other hand, the support $\mathcal{I}_{\bf C}$ of $\nu_Q$ intersects $\mathbb Z^{*-}$ and the support of $V'_{*-}$ equals $\mathbb Z^{*-}$;  hence $\nu_Q(V'_{*-})>0$.

 $\bullet$ in the case $\bf (Z, Z)$,  by using  Lemma \ref{MAJ+asym}$(a)$ again,
     \begin{align*}
        \sum_{j>n}r_j 
         &=\displaystyle \sum_{ x \leq -1 } \nu_Q(x) \mathbb{P}[\tau^S(x)>n]+   \sum_{x \geq 1}\nu_Q(x) \mathbb{P}[\tau^{S'}(x)>n] + \nu_Q(0) \mu_0(0)^n\\
         &\thicksim\frac{2c}{\sqrt{n}} \displaystyle \sum_{x \leq -1}\nu_Q(x) V_{*+}( \vert x\vert) +\frac{2c'}{\sqrt{n}} \displaystyle \sum_{x \geq 1}\nu_Q(x) V'_{*-}(x)=\ \frac{2(c\  \nu_Q({\check{V}_{*+}}) + c' \ \nu_Q(V'_{*-}))}{\sqrt{n}} 
     \end{align*} 
 with $0<    \nu_Q({\check{V}_{*+}}),  \nu_Q(V'_{*-})<+\infty$.

\item  Lemmas \ref{MAJ+asym} and \ref{estimate drift} yield the formulae for $\mathcal E(x, y)$.  Furthermore,  when $x\leq -1$ and $y\geq 0$, it holds

$\bullet$ in the case $\bf (Z, Z)$ \quad    $n^{3/2} Q_n(x, y) + \mathcal E(x, y) \preceq (1+\vert x\vert) \sum_{z>y} z\mu(z) \preceq  {1+|x|\over 1+y^{2+\delta_0}} \mathbb E[(\xi_1^+)^{3+\delta_0}].$

$\bullet$ in the case $\bf (P, Z)$ \quad    $n^{3/2} Q_n(x, y) + \mathcal E(x, y) \preceq (1+\vert x\vert) \sum_{z>y}  \mu(z) \preceq {1+\vert x\vert\over 1+y^{2+\delta_0}}\mathbb E[(\xi_1^+)^{2+\delta_0}],$

\noindent Similarly, when  $x\geq 1$ and $y\leq 0$, in both cases $\bf (Z, Z)$ and $\bf (P, Z)$, 
$$n^{3/2} Q_n(x, y) + \mathcal E(x, y) \preceq (1+x) \sum_{z<y} \vert z\vert \mu'(z) \preceq {1+x\over 1+\vert y\vert ^{2+\delta_0}}\mathbb E[({\xi_1'}^-)^{3+\delta_0}].
$$
At last,  for $y \neq 0$, $\displaystyle
Q_n(0, y)=\mu_0(0)^n \mu_0(y). 
$ Consequently, 
$$
\sup_{x \in \mathbb Z} {n^{3/2}Q_n(x, y) + \mathcal E(x, y) \over 1+\vert x\vert ^{1+\delta}}(1+\vert y\vert ^{1+\delta})\preceq {1\over 1+\vert y\vert^{1+\delta_0-\delta}}+ (1+\vert y\vert ^{1+\delta})\mu_0(y),
$$
where the right   side term is summable with respect to $y$. 
Since $n^{3/2}Q_n(x, y) \rightarrow \mathcal E(x, y)$ for any $x, y\in \mathbb Z$   and $0<\delta< \delta_0$, the dominated convergence theorem yields   $n^{3/2}Q_n  \rightarrow \mathcal E $  in $\mathcal L(\mathcal B_\Psi)$. 
 \end{enumerate}
\end{proof} 
Proposition \ref{QZZ} implies that  in the  cases ${\bf (P, Z)}$ and ${\bf (Z, Z)} $, the sequence $(Q_n)_{n \geq 1} $ is an aperiodic renewal sequence of operators on ${\mathcal B_\Psi}$ satisfying  {\it R4}$({L}, \beta)$  and  {\it R5}$({L}, \beta)$  with $\beta = 1/2$ and ${L}$ being constant. Hence, by using \cite[Theorem 1.4]{gouezel}, we obtain the following corollary, with   $\displaystyle T_n =  \sum_{\ell=1}^{+\infty}Q_n^{(\ell)}  $ introduced in \eqref{azevb}. 
 \begin{coro} \label{convergeH} 
The sequence $(\sqrt{n} T_n)_{n \geq 1}$ converges in $(\mathcal L({\mathcal B_\Psi}), \Vert \cdot \Vert_{\mathcal B_\Psi})$ to   the operator  ${\bf c}^{-1}\Pi_Q$ with   $\displaystyle  {\bf c}={ \pi\bigl(2c\nu_Q( \check{V}_{*+})+2c'\nu_Q(V'_{*-})\bigr)}$ in the case  ${\bf (Z, Z)} $  and 
 $\displaystyle  {\bf c}={\pi (2c')\nu_Q(V'_{*-})}$ in the case ${\bf (P, Z)}$.  As a consequence,  
for any $x, y \in \mathbb Z$,
$$\displaystyle \lim_{n \to +\infty} \sqrt{n}  T_n(x, y)={\bf c}^{-1}\nu_Q(y).
$$ 
\end{coro}

\subsection{Proof of Theorem \ref{theo1} in the recurrent cases \texorpdfstring{$\bf(P, Z)$}{Z} and \texorpdfstring{$\bf(Z, Z)$}{Z}}\label{secproofrecurentcase}

The proof is   based on  the   decomposition of trajectories
 \eqref{decompositiontraject}  which holds for $x, y \in \mathbb Z$.  In both cases $\bf (Z, Z)$ and $(\bf P, Z)$, the stopping times $C_k, k \geq 0$ are $\mathbb P$-as  finite  and  \eqref{decompositiontraject} may be written as  : for any $x\in \mathbb Z$,  
\begin{equation}\label{Pnxyrecurrent}
\mathbb P_x[X_n=y]= 
 \sum _{k=0}^n T_{n-k}  {\bf V}_{k, y}(x) 
\end{equation}
where the  functions $\displaystyle   {\bf V}_{k, y} $ are defined in  \eqref{hny} and  \eqref{h'ny}.
  In order to obtain the asymptotic behavior of $\mathbb P_x[X_n=y]$,  we use the following lemma which is a functional version of   \cite[Lemma 2.2]{iglehart}. 
 \begin{lemma}\label{qjksbh}
 Let $(P_n)_{n \geq 0}$  be  a sequence of positive operators acting on  a  Banach space $ ({\mathcal B}, \vert \cdot \vert_{\mathcal B })$  of real functions on $\mathbb R$ and $(g_n)_{n \geq 0}$ a sequence of positive functions in $\mathcal B$ such that
 \begin{enumerate}[(1)]
 \item the sequence $(\sqrt{ n}P_n)_{n \geq 0}$ converges in $ (\mathcal L(\mathcal B), \Vert.\Vert_{\mathcal{B} }) $ to $P$,
 \item the sequence $(n g_n)_{n \geq 0}$ is bounded in $\mathcal B$ and $\displaystyle \sum_{n \geq 0} \vert g_n\vert_\mathcal B<+\infty.$
 \end{enumerate}
 Then  the sequence  $\displaystyle \bigg(\sqrt{n} \sum_{k=0}^{n}P_{n-k}(g_{k})\bigg)_{n \geq 0}$ converges in $\mathcal B$ to $P(\mathbf G)$  with $\displaystyle \mathbf G:= \sum_{k \geq 0} g_k.$
  \end{lemma}
\begin{proof}
Fix $\varepsilon >0$ and $N$ such that $1\leq N \leq \lfloor n (1-\varepsilon)\rfloor$.   First, we decompose 
\begin{align*}
     \sum_{k=0}^{n}P_{n-k}(g_{k})= \sum_{k=0}^{N}+\sum_{k=N+1}^{ \lfloor n(1-\varepsilon)\rfloor}+\sum_{k=\lfloor n(1-\varepsilon)\rfloor +1}^{n}:=I_n+J_n+K_n
\end{align*}
\\
Since 
$\displaystyle \sqrt{n} I_n=\sum_{k=0}^{N }\sqrt{ n-k } P_{n-k}\left({\sqrt \frac{n}{n-k}}g_k \right)$, the sequence  $(\sqrt{n} I_n)_{n \geq 0}$ converges in $\mathcal B$ to $\displaystyle P\left(\sum_{k=0}^N g_k\right)$.
\\
Using the fact that $\Vert P_n\Vert_\mathcal B \leq A/\sqrt{n}$ for some strictly positive constant $A $, we obtain
\begin{align*}
 \vert \sqrt{n} J_n\vert_{\mathcal B} \leq A\sum_{k=N+1}^{ \lfloor n(1-\varepsilon)\rfloor}\sqrt{\frac{n}{n-k}}\vert g_k\vert_\mathcal B
\leq \frac{A}{\sqrt{\varepsilon}}\sum_{k=N+1}^{+\infty}\vert g_k\vert_\mathcal B.
\end{align*}
At last, since $\vert g_n\vert_\mathcal B\leq B / n $ for some  strictly positive constant $B $, we may write
\begin{align*}
 \vert \sqrt{n} K_n\vert_{\mathcal B}  &\leq AB\sqrt{n}  \sum_{k=\lfloor n(1-\varepsilon)\rfloor +1}^{n-1} \frac{1}{k\sqrt{ n-k }} + \Vert P_0\Vert_\mathcal B \ B/\sqrt{n}
\\
&\leq \frac{AB\sqrt{n}}{ n(1-\varepsilon) }\sum_{j=1}^{\lfloor n\varepsilon\rfloor}\frac{1}{\sqrt{j}}+ \Vert P_0\Vert_\mathcal B \ B/\sqrt{n}
\\
&\leq 2AB \frac{\sqrt{\varepsilon}}{1-\varepsilon}+ \Vert P_0\Vert_\mathcal B \ B/\sqrt{n}.
\end{align*}
Consequently,
\begin{align*}
\limsup_{n \to +\infty} \bigg\vert \sqrt{n} \sum_{k=0}^{n}P_{n-k}(g_{k}) - P(\mathbf G)\bigg\vert_{\mathcal B} 
&\leq \Vert P\Vert_\mathcal B \sum_{k =N+1}^{+\infty}\vert g_k\vert_\mathcal B+ \frac{A}{\sqrt{\varepsilon}}\sum_{k=N+1}^{+\infty}\vert g_k\vert_\mathcal B +
2AB \frac{\sqrt{\varepsilon}}{1-\varepsilon} 
\end{align*}
and one concludes  by  letting $ N\to +\infty$ then $\varepsilon \to 0$.
\end{proof}
 The convergence in Theorem \ref{theo1} in the null recurrent cases  ${\bf (Z, Z)} $ and ${\bf (P, Z)} $ follows by combining Corollary
 \ref{convergeH} and the previous lemma with $ \mathcal B= \mathcal B _\Psi,  P_n = T_n, P= {\bf c}^{-1}\Pi_Q, g_n=    {\bf V}_{n, y} $ and $\mathbf G= \displaystyle\sum_{n \geq 0}    {\bf V}_{n, y} $. This proves that, for any $x, y$ in $\mathbb Z$, the sequence  $(\sqrt{n} \mathbb P_x[X_n=y])_{n \geq 1}$ converges   to
$
 C_y=   \nu_Q\left(\sum_{n \geq 0}    {\bf V}_{n, y} \right)/\bf c 
$
 with  $\displaystyle  {\bf c}={ \pi\bigl(2c\nu_Q( \check{V}_{*+})+2c'\nu_Q(V'_{*-})\bigr)}$ in the case  ${\bf (Z, Z)} $  and 
 $\displaystyle  {\bf c}={\pi (2c')\nu_Q(V'_{*-})}$ in the case ${\bf (P, Z)}$. To obtain a more explicit expression of this   limit quantity $C_y$, we focus  on the case ${\bf (Z, Z)} $, the argument in the case ${\bf (P, Z)} $ is similar. Assume for instance  $y \leq -1$. On the one hand, by the ``balayage theorem", it holds
 \begin{align*}
 \nu_Q\left(\sum_{n \geq 0}    {\bf V}_{n, y} \right)&=\sum_{ x \leq -1}\nu_Q(x)  \sum_{n \geq 0}   \mathbb P [\tau^S(x)>n, x+S_n=y]  \\
 &
 =\sum_{ x \leq -1}\nu_Q(x)  \sum_{n \geq 0}   \mathbb P_x [C_1>n,  X_n=y] 
 \\
 &=\sum_{ x \leq -1}\nu_Q(x) \mathbb E_x[\sum_{n=0}^{C_1-1} {\bf 1}_{\{y\}}(X_n)]
 \\
 &= \mathbb E_{\nu_Q}[\sum_{n=0}^{C_1-1} {\bf 1}_{\{y\}}(X_n)]= \lambda_{\mathcal X}(y).
 \end{align*}  
  In order to identify the constant ${\bf c}$, we set 
$
T_n:= \inf\{0\leq k \leq n\mid   S_k=  \max(S_0, \ldots, S_n)\} 
$
 for $n \geq 0$ and decompose  
the quantity $\mathbb P[\tau^S(x)>n, x+S_n=y]$  as follows :
   \begin{align*}
\mathbb P [\tau^S(x)>n, x+S_n=y]&=
   \sum_{k=0}^{n}\sum_{M=x\vee y}^{-1}\mathbb P_x [\tau^S(x)>n, T_n=k, x+S_k=M, x+S_n=y] 
  \\
  &=\sum_{k=0}^{n}\sum_{M=x\vee y}^{-1}
  \mathbb P[x+S_1 < M, \ldots, x+S_{k-1}<M, x+S_k=M]
  \\
  &   \qquad \qquad\times \mathbb P[\xi_{k+1}\leq 0, \ldots, \xi_{k+1}+\ldots+ \xi_{n-1}\leq 0, \xi_{k+1}+\ldots +\xi_n=y-M]
  \\
 &=\sum_{k=0}^{n}\sum_{M=x\vee y}^{-1}
  \mathbb P[\xi_2+\ldots +\xi_k>0, \ldots, \xi_k>0, S_k=M-x]
  \\
  &   \qquad \qquad\qquad \qquad \times \mathbb P[S_{1}\leq 0, \ldots, S_{n-k-1}\leq 0, S_{n-k}=y-M]
  \\
  &=\sum_{k=0}^{n}\sum_{M=x\vee y}^{-1}
  \mathbb P[S_1>0, \ldots, S_{k-1}>0, S_k=M-x]
  \\
  &   \qquad \qquad \qquad \qquad \times \mathbb P[S_{1}\leq 0, \ldots, S_{n-k-1}\leq 0, S_{n-k}=y-M] 
  \\
  &=\sum_{k=0}^{n}\sum_{M=x\vee y}^{-1} \mathbb P[\tau_-^S>k, S_k=M-x] \times \mathbb P[\tau_{*+}^S>n-k, S_n-k=y-M]
 \end{align*}
 which yields, for $x, y \leq -1$, 
\[
 \sum_{n=0}^{+\infty}{\bf V}_{n, y}(x)
 = 
 \sum_{M=x\vee y}^{-1} \left( \sum_{n=0}^{+\infty}\mathbb P[\tau_-^S>n, S_n=M-x]\right)  \left( \sum_{n=0}^{+\infty}\mathbb P[\tau_{*+}^S>n, S_n=y-M]\right). 
\]
Recall that, by   a similar reversing time argument, it holds  $\displaystyle \sum_{n=0}^{+\infty}\mathbb P[\tau_-^S>n, S_n=M-x]=U_{*+}(M-x)$ and $
 \displaystyle \sum_{n=0}^{+\infty}\mathbb P[\tau_{*+}^S>n, S_n=y-M]=U_-(y-M)$. 
 Hence $\displaystyle \sum_{n=0}^{+\infty}{\bf V}_{n, y}(x) = \sum_{M=x\vee y}^{-1}U_{*+}(M-x) U_-(y-M) $, 
 which yields for  $y \leq x\leq -1$,
 \[
 \lambda_{\mathcal X}(y)= \nu_Q\left(\sum_{n=0}^{+\infty}{\bf V}_{n, y}\right)=
 \sum_{x \leq -1}   \nu_Q(x) \left( \sum_{M=x}^{-1}U_{*+}(M-x)U_-(y-M) \right).
 \]
 The function $z\mapsto U_-(z)$ is bounded  and satisfies $\displaystyle \lim_{y \to -\infty}U_-(y-M)= \frac{1}{\mathbb E[S_{\tau^S_-}]}$ for any $M\leq -1$, by renewal theorem. The dominated convergence theorem thus implies that $\displaystyle \lambda_{\mathcal X}(-\infty):= \lim_{y\to -\infty}\lambda_{\mathcal X}(y)$ does exists and
\[
 \nu_Q(\check{V}_{*+})=   \sum_{x \leq -1} \nu_Q(x) \underbrace{\left(\sum_{M=x}^{-1}U_{*+}(M-x)\right)}_{=V_{*+}(x)}= \lambda_{\mathcal X}(-\infty)\mathbb E[S_{\tau^S_-}].
\]
 Similarly, $\displaystyle \lambda_{\mathcal X}(+\infty):= \lim_{y\to +\infty}\lambda_{\mathcal X}(y)$ does exists and
$\displaystyle
 \nu_Q(V'_{*-})=     \lambda_{\mathcal X}(+\infty)\mathbb E[S_{\tau^{S'}_+}].
$   Finally, combining these two last equalities with \eqref{kjzghv} and \eqref{jhevgt},  we obtain ${\bf c}=\sqrt{\frac{\pi}{2}}(\sigma \lambda_{\mathcal X} (-\infty)+\sigma' \lambda_{\mathcal X} (+\infty))$ in the case ${\bf (Z, Z)} $  and ${\bf c}= \sigma' \lambda_{\mathcal X} (+\infty))$ in the case ${\bf (P, Z)}.$ 
 This achieves the proof of Theorem \ref{theo1} in the null recurrent case.
 \section{The transient case \texorpdfstring{$\bf (Z, P)$}{Z}} \label{sectiontransientZP}

This is the simplest transient case to study since $\rho_\mathcal X =1$. We still set $\Psi(x)= 1+\vert x\vert^{1+\delta}$ with $0<\delta < \delta_0.$

  \subsection{Preliminary remarks}
  
 Notice first that in this case, $\mathbb P_x[C_1<+\infty]=1$ for any $x \leq  0$  while  when $x \geq 1$,  
 $$\mathbb P_x[C_1<+\infty] \leq  \mathbb P_1[C_1<+\infty]<1.$$
 Consequently, the  switching operator $Q$ is submarkovian in this case.
This forces us, by using   the classical ``Doob's $H$-transform"  in probability theory, to slightly modify  the operator $Q$ and obtain a markovian  operator $\ {^HQ}$.

 The asymptotic behavior  of the quantities $\mathbb P_x[C_1=n, X_n=y] $ in the  case $(\bf Z, P)$ are given by Lemmas \ref{MAJ+asym} and \ref{relativisation} applied respectively to $S$ and $S'$ (with $\lambda'<0$ since $\mathbb E[\xi'_1]>0$). Let us recall them briefly for $x \neq 0$.

\noindent $\bullet$ \quad  \underline{when $x \leq -1$ and $y \geq 0$}, 
 
\begin{equation} \label{MAJlocalS}
\mathbb{P}_x[C_1=n,  X_n =y] \preceq \frac{(1+\vert x\vert)\sum_{z> y} z\mu(z)}{n^{3/2}}
\end{equation}
and
\begin{equation} \label{asymplocalS}
\mathbb{P}_x[C_1=n,  X_n =y]  \thicksim   
 \frac{1}{\sigma \sqrt{2\pi}} \frac{V_{*+}(\vert x\vert)}{n^{3/2}}\displaystyle\sum_{w \geq 1} V_{-}(w) \mu(  y+w);
\end{equation}

\noindent $\bullet$ \quad  \underline{when  $x \geq 1$ and $y \leq 0$},   by applying \eqref{zqhcgbef}  with   $p= 5/2+\delta_0$ and $\epsilon=\delta  <\min(1,\delta_0)$, one obtains
\begin{equation}\label{ZPS'}
\mathbb{P}_x[C_1=n,  X_n =y] 
 \preceq\,   {1\over \Psi(y)\ n^{3/2+\delta_0-\delta}} \quad {\rm  with} \ \delta_0-\delta >0.
\end{equation}

 \noindent $\bullet$   
$\quad \displaystyle
\mathbb{P}_0[C_1=n,  X_n =y] 
 =\mu_0(0)^{n-1} \mu_0(y) 
$ with $\mu_0(0) <1$.

The following statement is an adaptation of the  Proposition \ref{QZZ} to the transient cases with some major modifications due to the fact that the operator $Q$ is submarkovian here. As  in the previous section, this operator $Q$  may be decomposed  as $Q=\displaystyle \sum_{n \ge 1} Q_n$ with $Q_n\varphi(x) = \mathbb E[C_1=n, \varphi(X_n)]$ for any $\varphi \in \mathcal{B}_{\Psi}$ and  we set $Q(z)= \displaystyle\sum_{n \ge 1} z^n Q_n$ for any $\vert z\vert \leq 1$.

\begin{prop} \label{QZP} \underline{Case $  (\bf Z, P)$}.

Assume hypotheses {\bf O1}-- {\bf O4}  and moment conditions  $\bf (Z, P)$. Then,  

(i)  the map $Q$ acts on $  \mathcal B_{\Psi}$ and $Q( \mathcal B_{\Psi}) \subset \mathbb L^\infty(\mathbb{Z})$;

(ii)  $Q$ is a compact operator on $ \mathcal B_{\Psi} $ with spectral radius $\rho_{\Psi}  $ strictly less than 1 and  there exists a positive bounded function $H$ in $\mathcal B_\Psi$ such that 
$QH= \rho_\Psi H$ and $\displaystyle\inf_{x \in \mathbb Z} H(x) >0$.
    
 Let $\ {^HQ}$ and $\  ^H{Q}_n $ for $ n \geq 1$ be the operators  defined by : for any $\varphi \in \mathcal B_{\Psi}$, 
 \begin{equation} \label{widetildeQ}
 \ {^HQ}\varphi = \frac{1}{ \rho_\Psi} \frac{1}{ H}Q(H\varphi) \quad and \quad 
  ^H{Q}_n\varphi = \frac{1}{ \rho_\Psi} \frac{1}{ H}Q_n(H\varphi).
 \end{equation}
 
 Then, the operator $\displaystyle \ {^HQ}: \varphi \mapsto \frac{1}{ \rho_\Psi} \frac{1}{ H}Q(H\varphi)$  is a markovian and compact operator   on $ \mathcal B_{\Psi}$   with  the unique and simple dominant eigenvalue $1$. Hence, it can be decomposed as
 $
\ {^HQ}= {^H \Pi} + Q'
 $
where 
\\
\indent $\bullet$ \ \ ${^H \Pi}$ is the eigenprojector from $\mathcal B_{\Psi}$ to $\mathbb C \bf{1}$ corresponding to the eigenvalue $1$ and ${^H \Pi}(\varphi)=  {^H\nu}(\varphi) \mathbf{1}$, where $^H\nu$ is the unique $\ {^HQ}$-invariant probability measure on $\mathbb Z$;

\indent $\bullet$ \ \ the spectral radius of $ Q'$  on $ \mathcal B_{\Psi}$ is strictly less than 1; 

 $\bullet$ \ \ ${^H \Pi} Q'= Q'{^H \Pi} = 0$.

\noindent  Furthermore, the spectral radius  of $^HQ(z)= \sum_{n \geq 1} {z^n}  ^HQ_n $ on $ \mathcal B_{\Psi}$  is strictly less than 1 for $z \in \overline{\mathbb{D}}\setminus\{1\}$ and the sequence $(^H{Q}_n)_{n\geq 1}$ holds the following properties:
\begin{enumerate}[(a)]
\item the operators $\  ^H{Q}_n$ act  on $\mathcal B_{\Psi}$ and 
 $\displaystyle \Vert   ^H{Q}_n \Vert_{ \mathcal B_{\Psi}} = O\left(\dfrac{1}{n^{3/2}}\right)$;  
       \item the sequence $(n^{3/2} {^H}Q_n)_{n \geq 0}$ converges in $\mathcal  L( \mathcal B_{\Psi})$  to the operator $ ^H\mathcal{E}$  whose transition matrix is given by: for any $x, y \in \mathbb Z$, 
\[
{^H {\mathcal E}}(x, y) = 
\left\{
\begin{array}{cl}
\frac{ H(y)}{ \rho_\Psi H(x)} \left(\frac{1}{\sigma\sqrt{2\pi}}  V_{*+}(\vert x\vert)   \displaystyle\sum_{z \geq 1} V_{-}(z) \mu (  y+z)\right)
 & if\  \Bigl(x\leq  -1 \  and  \ y \geq 0\Bigr),  
 \\
0& if \ \Bigl(x=0\ and \  y \neq 0\Bigr) \ \\
\ & \qquad \qquad or \ \Bigl(x\geq 1\ and \  y\leq 0\Bigr).
\end{array}
\right.
\] 
  \end{enumerate}
\end{prop}
\begin{remark} Notice that by the definition \eqref{widetildeQ},  for any $\ell \geq 1$, the iterates  $Q^{(\ell)}$ and $^H Q^{(\ell)}$ satisfy:  for any $ \varphi  \in \mathcal B_\Psi$, 
\begin{equation*} 
{^HQ}  ^{(\ell)} ( \varphi) =\frac{1}{ \rho^\ell_\Psi} \frac{1}{ H}Q^{(\ell)}(H\varphi) 
\end{equation*} 
and
\begin{equation} \label{iteratesHQ}
 ^H{Q}_n ^{(\ell)}(\varphi):= \sum_{j_1+ \ldots + j_\ell=n}{^HQ}_{j_1}\cdots {^HQ}_{j_\ell}(\varphi)= 
 \frac{1}{ \rho^\ell_\Psi} \frac{1}{H}Q^{(\ell)}_n(H\varphi).
\end{equation} 
\end{remark}
The proof  of Proposition \ref{QZP} is closed to the one of Proposition \ref{QZZ}; we insist here on  the main changes  due to the fact that $Q$ is submarkovian. 
For any $\varphi \in \mathcal{B}_{\Psi}$ and $x \geq 1$,  formula \eqref{Qswitchingtime-finite2} yields 
     \begin{align*}
         |Q\varphi(x)| &\leq \ \sum_{y\leq 0}   \sum_{t = -x+1}^{0}{\mathcal U}'_{d^*}(t) \mu'_{d^*}( y-x-t)\, \vert  \varphi(y)|
         \\ 
          &\leq \ \sum_{ y\leq 0 }  \sum_{t = -x+1}^{0} \mu'_{d^*}( y-x-t)\, \vert  \varphi(y)|
         \\
         &\leq  \sum_{y \leq 0} \vert\varphi(y)\vert\ \mu'_{d^*}(-\infty, y). 
     \end{align*} 
     Hence 
     $\displaystyle  |Q\varphi(x)| \leq    \vert \varphi\vert_{\Psi}\ \sum_{ y\leq 0 }  \Psi(y) \ \mu'_{d^*}(-\infty, y) 
     $
 \quad  with \quad $\displaystyle \sum_{ y\leq 0 }  \Psi(y) \ \mu'_{d^*}(-\infty, y)
<+\infty$ 
 \cite[Theorem 7.7.13]{Alsmeyer}.   
  When $x \leq 0$,   
   we use  the  inequality \eqref{bound}. Finally,    $Q(\mathcal B_\Psi)\subset \mathbb L^{\infty}(\mathbb Z)$    and  $Q$ is  compact on ${\mathcal B_\Psi}$. 
     
The spectral radius $ \rho_{\Psi}$ of $Q$ on $\mathcal B_\Psi$ is less than $  1$ since the power sequence $(Q^n)_{n\geq 1}$ is bounded in  $\mathcal L(\mathcal B_{\Psi})$; indeed, for any $\varphi \in \mathcal B_\Psi, n \geq 1$ and $x \in \mathbb{Z}$,
    \begin{align*}
        \vert Q^{n}\varphi(x) \vert \leq \displaystyle \sum_{y\in \mathbb{Z}} Q^{n-1}(x,y) \vert Q\varphi(y) \vert \leq \vert Q\varphi \vert_{\infty} \displaystyle \sum_{y\in \mathbb{Z}} Q^{n-1}(x,y) \leq\vert Q\varphi \vert_{\infty}.
    \end{align*}
 Hence $\Vert Q^n \Vert_{ \mathcal B_{\Psi}}  \leq  
    \vert Q^n  \Psi  \vert_\infty
    \leq \vert Q\Psi \vert_{\infty}$ for any $n \geq 1$ which yields
 $ \rho_{\Psi} \leq 1.$  
 
 Under condition {\bf O3}, the operator  $Q$ is strictly positive on the subspace $\mathcal B_\Psi(\mathcal I_{\bf C})\subset \mathcal B_\Psi$ of functions $\varphi \in \mathcal B_\psi$ with support included in $\mathcal I_{\bf C}$; in other words,  $Q\varphi(x)>0$ for any $x \in \mathcal I_{\bf C}$  and any non negative function $\varphi\neq 0$  in $\mathcal B_\Psi(\mathcal I_{\bf C})$.  By Krein-Rutman  theorem, there exists  a unique positive function $ H' \in \mathcal B_\Psi(\mathcal I_{\bf C}) $   such that $  Q H' = \rho_\Psi H'$.     Since $Q$ is bounded from $\mathcal B_\Psi$ to $\mathbb L^\infty(\mathbb Z)$,  the function $ H' $ is bounded on $\mathcal I_{\bf C}$.

When  $\mathcal I_{\bf C} =\mathbb Z$, the function $H'$ is positive and bounded on $\mathbb Z$, we note  $H'=H$    in the sequel to simplify. 

 When $\mathcal I_{\bf C} \subsetneq \mathbb Z$,  since for any function $\varphi: \mathbb Z \to \mathbb R^+$   the values of $Q\varphi$ depend  only on the values of $\varphi$ on $ \mathcal I_{\bf C}$, we may  extend $H'$ to the whole line $\mathbb Z$ by setting
$\displaystyle H(x)=  {1\over \rho_\Psi} \sum_{y \in \mathcal I_{\bf C}} H'(y)Q(x, y)$ for any $x\in \mathbb Z$.
The function $H$ is  positive and bounded on $\mathbb Z$ and satisfies  $ Q H= \rho_\Psi H$ on the whole line $\mathbb Z$.

As a first consequence, we obtain  that $H$ is also bounded from below by a strictly positive constant. Indeed, $H(x) \geq Q(x, 0)H'(0)/\rho_\Psi$ with $H'(0) >0$ since $0 \in \mathcal I_{\bf C}$ and $\displaystyle \inf_{x \in \mathbb Z} Q(x, 0) >0$, by the expression \eqref{Qswitchingtime-finite2}  and  classical results on random walks (see  e.g. \cite[Lemma 5.3.3]{peignevo}).

 As a second consequence, we obtain $\rho_\Psi <1$. Indeed,  the condition $\mathbb E[\xi'_1]>0$   yields  $\mathbb P_x[C_1=+\infty] \geq \mathbb P_1[C_1=+\infty]=:\epsilon_0>0$ for any $x \geq 1$.
When $x \leq 0$, noticing that $\mathbb P_x[C_1<+\infty]=1$,  we obtain 
$
\mathbb P_x[C_2=+\infty]\geq \inf_{y \geq 1}\mathbb P_y[C_1=+\infty]= \epsilon_0>0.
$
Consequently $\displaystyle \vert Q^2{\bf 1}\vert_\infty = \sup_{x \in \mathbb Z} \mathbb P_x[C_2<+\infty]  < 1-\epsilon_0 $ which yields  $\rho_\Psi<1$ since $ Q^2H = \rho_{\Psi}^2 H$ and $H$ is bounded.

 The control of  the peripheral spectrum of $Q$   comes from  the one of $\ {^HQ}$.  This operator $\ {^HQ}$   is   compact and markovian  on $\mathcal B_\Psi$ and there is in particular a one-to-one correspondance between the set of eigenvalues  of $Q$ with modulus 1  (and  the corresponding eigenspaces) and the ones of $\ {^HQ}$. Let   $\theta \in \mathbb{R}$ and $\psi \in {\mathcal B_\Psi}$ such that $\ {^HQ}\psi =   e^{i\theta}\psi$. Using the same argument as in Proposition \ref{QZZ} and the irreducibility of $\ {^HQ}$ on $\mathcal I_{\bf C}$,  we prove that $1$ is the unique eigenfunction of $\ {^HQ}$  with modulus 1 and that the corresponding eigenfunctions are constant. Equivalently, the peripheral spectrum of $Q$  is   reduced to $\rho_\Psi$ and the eigenspace of $\rho_\Psi$ equals to $\mathbb C H$. Similarly, we prove that  for any $z\in \overline{\mathbb{D}}\setminus\{1\}$, the spectral radius of   $\ {^HQ}(z)$ is strictly less than $ 1$.

 The control of the behavior of the sequence $( {^H} Q _n)_{n \geq 1}$   in  the case  $ \bf(Z, P)$ is a direct consequence of   \eqref{MAJlocalS}, \eqref{asymplocalS}  and  \eqref{ZPS'}:   
  
 $\bullet$ for $x \leq -1$  and $y \geq 0$, 
 $ \quad   ^H{Q}_n(x, y)  \preceq  (1+ \vert x\vert)     \left(\sum_{z >y} z  \mu (z)\right)  \frac{1}{n^{3/2}}
\quad $
 and 
\begin{align*}
  ^H{Q}_n(x, y)&\sim \frac{ H(y)}{ \rho_\Psi H(x)} \left(\frac{1}{\sigma\sqrt{2\pi}}  V_{*+}(\vert x\vert)   \sum_{w \geq 1} V_{-}(w) \mu (  y+w)\right) \frac{1}{n^{3/2}}.
 \end{align*}

 $\bullet$ for $x \geq  1$  and $y \leq 0$,    
 $\displaystyle \quad   ^H{Q}_n(x, y)   \preceq\,   {1\over \Psi(y) n^{3/2+\delta_0-\delta}} \quad {\rm  with} \ \delta_0-\delta >0;
 $

 $\bullet\,   
  \displaystyle
^H{Q}_n(0, y) 
\preceq \mu_0(0)^{n-1} \mu_0(y) 
$ with $\mu_0(0) <1$.

Finally, as in  Proposition \ref{QZZ},  the sequence $(Q_n)_{n \geq 0}$ converges in $\mathcal L(\mathcal B_\Psi)$ to  some operator $\mathcal E$, hence   $(^HQ_n)_{n \geq 0}$ converges  to  $^H\mathcal E$ and a direct computation gives the explicit formula of $^H\mathcal E$. This proves assertion {\it (b)}.

\subsection{Proof of Theorem \ref{theo1} in the case  \texorpdfstring{$\bf  (Z, P)$}{Z}}

The proof is   based on  the   decomposition of trajectories
 \eqref{decompositiontraject}   and  \eqref{iteratesHQ} :  for any $x, y\in \mathbb Z$

\begin{equation}\label{returnTransientcase}
\mathbb P_x[X_n=y]
 = 
 H(x)\sum_{\ell=0}^{+\infty} \rho_\Psi^\ell  \sum _{k=0}^n \  {^HQ}_k^{(\ell)} (  {\bf V}_{n-k, y}/H)(x)
\end{equation}
where the functions $\displaystyle   {\bf V}_{k, y}$ for $k\geq 1, y \in \mathbb Z$ are defined in \eqref{hny} and  \eqref{h'ny}. 
We use the following lemma which is a functional version of   \cite[Lemma II.8]{lepagepeigne}. 
 \begin{lemma} \label{zb}
 Let $(P_n)_{n \geq 0}$  be  a sequence of positive operators acting on  a  Banach space $ ({\mathcal B}, \vert \cdot \vert_{\mathcal B })$  of real functions on $\mathbb R$ and $(g_n)_{n \geq 0}$ a sequence of positive functions in $\mathcal B$ such that
 \begin{enumerate}[(1)]
 \item the sequence $(n^{3/2}P_n)_{n \geq 0}$ converges in $ (\mathcal L(\mathcal B), \Vert.\Vert_{\mathcal B})$ to $P$,
 \item the sequence $(n^{3/2} g_n)_{n \geq 0}$ converges in $\mathcal B$ to some element $g$.
 \end{enumerate}
 Then   the sequence  $\displaystyle \bigg(n^{3/2} \sum_{k=0}^{n}P_{k}(g_{n-k})\bigg)_{n \geq 0}$ converges in $\mathcal B$ to $\mathbf{P} (g)+ P(\mathbf{G})$  with $\displaystyle \mathbf{P}:= \sum_{k \geq 0} P_k$ and $\displaystyle \mathbf{G}:= \sum_{k \geq 0} g_k.$ 
  \end{lemma}
\begin{proof}
Fix $n \geq 1$ and $1\leq i \leq n/2$.  For convenience, we decompose 
$$ \sum_{k=0}^{n}P_{k}(g_{n-k})=\sum_{k=0}^{i}+\sum_{k=i+1}^{ n-i-1}+\sum_{k=n-i}^{n}:= I_n+J_n+K_n.$$ 
\\
On the one hand, setting $\displaystyle a:= \sup_{k\geq 1} k^{3/2} \Vert P_k\Vert_{\mathcal B}$
and  $\displaystyle b:= \sup_{k\geq 0} k^{3/2} \vert  g_k\vert_\mathcal B$, it holds that  
\begin{align*}
\vert n^{3/2} J_n\vert_\mathcal B &\leq ab \  n^{3/2} \sum_{k = i+1}^{n-i-1} {1\over k^{3/2}(n-k)^{3/2}}
\\
&\leq  ab \ n^{3/2} \sum_{k = i+1}^{\lfloor n/2\rfloor} {1\over k^{3/2}(n-k)^{3/2}}
\\
& \preceq  \sum_{k = i+1}^{\lfloor n/2\rfloor} {1\over k^{3/2}} \preceq \frac{1}{\sqrt{i}}.
\end{align*}
On the other hand,
$\displaystyle \quad 
\vert n^{3/2}I_n- \mathbf P(g)\vert_\mathcal B
\leq \sum_{k =0}^i  \Vert P_k\Vert_{\mathcal B} \ \vert n^{3/2} g_{n-k}-g\vert_\mathcal B+
 \vert g\vert_{\mathcal B} \times 
 \sum_{k = i+1}^{+\infty}  \Vert P_k  \Vert_{\mathcal B},
$
with $\displaystyle \sum_{k = i+1}^{+\infty}  \Vert P_k  \Vert_{\mathcal B} \leq a \sum_{k = i+1}^{+\infty} k^{-3/2}\preceq \frac{1}{\sqrt{i}}$.
Similarly, 
\begin{align*}
\vert n^{3/2}K_n-   P(\mathbf G)\vert_\mathcal B
\leq \sum_{k =0}^i  \  \Vert n^{3/2} P_{n-k}-P\Vert _{\mathcal B} \ \vert g_k\vert_\mathcal B+  \Vert P\Vert_{\mathcal B}\times \sum_{k = i+1}^{+\infty} \vert g_k  \vert_{ \mathcal B } 
\end{align*}
with 
$\displaystyle 
  \sum_{k = i+1}^{+\infty} \vert g_k  \vert_{ \mathcal B }\leq   b \sum_{k = i+1}^{+\infty} k^{-3/2}\preceq \frac{1}{\sqrt{i}}.
  $
Consequently, we obtain
\[ 
\limsup_{n \to +\infty}\bigg\vert \left(n^{3/2}   \sum_{k=0}^{n}P_{n-k}(g_{k})\right)- (\mathbf{P} (g)+ P(\mathbf{G}))\bigg\vert_{\mathcal B} 
\preceq {1\over \sqrt{i}}
\]
and one concludes by  letting $i \to +\infty$.
 \end{proof}
We achieve the proof of Theorem \ref{theo1} in the  case ${\bf (Z, P)} $  by   combining  Theorem \ref{theogouezel+}, Proposition \ref{QZP}
  and Lemma \ref{zb} with $ \mathcal B= \mathcal{B}_{\Psi}, P_n ={^H}{Q}_n^{(\ell)} $ and $ g_n=   {\bf V}_{n, y}$. Indeed, by Proposition  \ref{QZP},  for any $\ell \geq 1$, the sequence  $\bigg({^HQ}_n\bigg)_{n\geq 1}$ satisfies the hypotheses of Theorem \ref{theogouezel+} (with $L=1$ and $\beta=1/2$)   so that  $\bigg(n^{3/2}\  ^H{Q}_n^{(\ell)}\bigg)_{n\geq 1}$ converges in $\mathcal L(\mathcal B_\Psi)$ to $\displaystyle  {^H}\mathcal E _\ell:= \sum_{i=0}^{\ell-1}  {^H}Q^{(i)} {^H}\mathcal E    {^H}Q^{(\ell-i-1)}$.  Similarly, for any $y \in \mathbb Z$, the sequence   $(n^{3/2}  {\bf V}_{n, y})_{n \geq 1}$  converges  in $\mathcal B_\Psi$ to the function $  {\bf V}_y$  defined in \eqref{hy} when $y\leq 0$ and  to $0$  when $y \geq  1$  since  $\rho'<1$.
 We may thus apply the previous lemma with $P={^H}\mathcal E _\ell, \mathbf P= \displaystyle\sum_{n\geq 0} {^H}{Q}_n^{(\ell)}$ and ${\bf G}= \displaystyle\sum_{n\geq 0}    {\bf V}_{n, y}$.

Finally, by  \eqref{evgj}, we may apply the dominated convergence theorem which yields 
\begin{align*}
\lim_{n \to +\infty} n^{3/2}  \mathbb P_x[X_n=y]&=
\lim_{n \to +\infty}H(x)\sum_{\ell=0}^{+\infty} \rho_\Psi^\ell
n^{3/2}\left(\sum _{k=0}^n    {^HQ}_k^{(\ell)} (  {\bf V}_{n-k, y}/H)(x)\right)
\\
&= H(x)\sum_{\ell=0}^{+\infty} \rho_\Psi^\ell \lim_{n \to +\infty} n^{3/2}\left(\sum _{k=0}^n     {^HQ}_k^{(\ell)} (  {\bf V}_{n-k, y}/H)(x)\right)
\\
&  = 
 H(x)\sum_{\ell=0}^{+\infty}  \rho_\Psi^\ell \ 
 \left(\sum_{n\geq 0} {^H}Q_n^{(\ell)}(  {\bf V}_y/H)(x)+
\sum_{n\geq 0}   {^H}\mathcal E_\ell(  {\bf V}_{n, y}/H)(x)
 \right)>0.
 \end{align*}


 \section{Proof of Theorem \ref{theo1}:  the transient cases (N, P) and (P, P) } 
 \label{sectiontransientNPPP}

 {\bf From now on, we assume that $\mu_0=\mu$. }  The  
 resulting process $\mathcal X$ oscillates between the two media $\mathbb Z^{-}$ and 
 $\mathbb Z^{*+}$. Obviously, this forces us to reformulate the sequence of 
 the switching times as follows:   $C_0=0$ and  for any $k \geq 0$, 
\begin{align}\label{ switchingtime}
    C_{k+1} :=\left\{ 
    \begin{array}{lll}
 \vspace{2mm}
    \inf \{n > C_k \mid  X_{C_k} + (\xi_{C_k + 1} +  
    \dotsi +\xi_{n}) \geq 1\} &\text{\rm if }  X_{C_k} \leq 0, \\
     \\
 \inf \{n > C_k \mid   X_{C_k} + (\xi'_{C_k + 1} + \dotsi + \xi'_{n})
  \leq 0 \} &\text{\rm if } X_{C_k} \geq 1. 
 \end{array}
 \right.
\end{align}  
 The study of the transient cases relies on  Theorem \ref{theogouezel+} 
 and requires a change of the measures $\mu$ and/or $\mu'$ to bring us back 
 to the cases already studied. 
 We first consider the case $(\bf N, P)$ which 
 is easier.  The most difficult 
 case  is $ {\bf (P, P)}$  where the change of measure 
 is more subtle.

 \subsection{The transient case  \texorpdfstring{$\bf  (N,P)$}{N}}\label{subsectionNP}

 As in section \ref{sectiontransientZP}, the proof is   based on  
 the   decomposition of trajectories \eqref{decompositiontraject} and 
  again, we have to take into account the fact that both random walks $S$ 
  and $S'$ are not centered. In particular,  we may use
   Lemma \ref{relativisation} with $\tau^S(x)$ in this case defined by 
  $\tau^S(x):= \inf\{n\geq 1\mid x+S_n \geq 1\}$ when $x \leq 0$ and $y \geq 1,$
   since we consider $\mu_0=\mu$; we   obtain
\begin{equation*}
\mathbb{P}_x[C_1=n,  X_n =y] 
 \preceq\,   (1+\vert x\vert) e^{\lambda  x} \left(e^{-\lambda y} 
 \sum_{z>y} \vert z\vert  e^{\lambda z}\mu (z)\right)\frac{ {\rho }^{n }}{n^{3/2}} 
\end{equation*}
and 
 \begin{equation*}
\mathbb{P}_x[C_1=n,  X_n =y]  
\thicksim\frac{1}{\overset{\circ}{\sigma}  \sqrt{2\pi}}\, 
\overset{\circ}{\ \ V_{*+}}(\vert x\vert) \  e^{\lambda x}
 \left(\sum_{w \geq  1}\overset{\circ}{\ V_-} (  w)e^{\lambda  w} 
 \mu  (y+w)\right)\frac{{\rho }^{n-1}}{n^{3/2}} 
\end{equation*}
  where $\lambda >0$, $\rho=L(\lambda) <1$, $\overset{\circ}{\sigma }  $ is
   the standard deviation of the centered random walk $ \overset{\circ}{S }$ derived
    from $S$ and $\overset{\circ}{\ \  V_{*+}} $ (resp. $\overset{\circ}{\ V_-} $) is its 
    strongly ascending (resp. weakly descending) renewal function. Similar estimates
     hold when  $x \geq 1$ and $y \leq 0$, replacing respectively $\mu, \lambda$ and 
     $\rho$ by $\mu', \lambda'<0$ and $\rho'=L'(\lambda')$.

The classical  walks $S$ and $S'$ converge respectively to $-\infty$ and $+\infty$ and  
play here a symmetric role; thus, without loss of generality,
  \underline{we may  assume in this subsection   that $\rho\geq \rho'$ and set 
 $  r = \rho'/\rho.$} 
 
The moment conditions in this case yield us to consider the   Banach space 
 $\mathcal B_\Psi$  
$$
  \ \Psi(x)= \left\{\begin{array}{lll}
e^{ (\vert  \lambda'\vert+\delta) \vert x\vert} & \text{if} & x \leq 0,\\
e^{ (\lambda+\delta) x}& \text{if} & x \geq 1 
\end{array}\right.$$
where   $\delta>0$.  
As above, the operator $Q$ may be decomposed  as $Q= \sum_{n \ge 1} Q_n$ 
with $Q_n\varphi(x) = \mathbb E_x[C_1=n, \varphi(X_n)]$. Notice that,  
by \eqref{valeurrelativisation},  
 there exist here  two strictly positive constants $C_{x, y}$ and $ C'_{x,y}  $  such that 
 $Q_n(x, y) \sim C_{x, y} \rho^{n-1}/n^{3/2}$ when $x \leq 0$ and $ y\geq 1$  and
   $Q_n(x, y) \sim  C'_{x, y} {\rho'}^{n-1}/n^{3/2}$ when  $x \geq 1$ and $ y\leq 0$. 

The strategy here is to apply Theorem   \ref{theogouezel+}. The proof is 
decomposed into several steps. 

$\bullet$ \underline{Extraction of the factor $\max(\rho, \rho')^n$} 

 We first introduce the sequence of operators $(\tilde{Q}_n)_{n \geq 1}$ 
  defined by:  for $n\geq 1$, 
\begin{equation}\label{Qnr}
\tilde{Q}_n(x, y):=
\left\{
\begin{array}{cccc}
 \mathbb P[\tau^{\overset{\circ}{S_{\ }} }(x)=n, 
 x+{\overset{\circ} {S_n}}=y]& {\rm if} &   \Bigl(x \leq 0 \ {\rm and}  \    y\geq 1\Bigr),
 \\
 r^n    e^{(\lambda'-\lambda)(x-y)}  
 \mathbb P[\tau^{\overset{\circ} {S'}}(x)=n, x+{\overset{\circ} {S'_n}}=y] 
 & {\rm if} & \Bigl(x\geq  1  \ {\rm and} \   y\leq 0\Bigr).\\
\end{array}
\right.  
\end{equation}
 
 Notice that $(\lambda'-\lambda)(x-y)<0$ for any $x\geq 1$ and $y\leq 0$.

This choice is governed by the following equality: for any $x, y \in \mathbb Z$ 
and $n \geq 1$,
\[
Q_n(x, y) = \rho^n e^{\lambda (x-y)}  \tilde{Q}_n(x, y) 
\]
  which can be rewritten as $Q_n\varphi(x)/\rho^n= e^{\lambda x}  \tilde{Q}_n 
  (\varphi e^{-\lambda \cdot})(x)$ for any  function $\varphi\in \mathcal B_\Psi$.  
   Hence,   up to the factor $\rho^n$, the behavior of  the quantity $Q_n(x, y)$ as
    $n\to +\infty$  is controlled by the one of  $\tilde{Q}_n(x, y)$.
  
   By iteration, for any $\ell \geq 1$, the operator
  $\displaystyle Q_n^{(\ell)}:= \sum_{j_1+ \ldots + j_\ell=n}Q_{j_1}\ldots Q_{j_\ell}$ 
  satisfies a similar property:  
$\displaystyle 
  Q_n^{(\ell)}\varphi(x)= \rho^n e^{\lambda x} \tilde{Q}_n^{(\ell)}(\varphi /e^{\lambda 
  \cdot})(x)$
  where $\displaystyle \tilde Q_n^{(\ell)}:= \sum_{j_1+ \ldots + j_\ell=n} 
  \tilde Q_{j_1}\ldots \tilde Q_{j_\ell}.$
 
 $\bullet$ \underline{The Doob's transform}
 
 The rest of the proof is similar to the argument developed in the case  ${\bf (Z, P)}$. 
By a straightforward computation, we can check that each operator $\tilde{Q}_n$ acts 
on $\mathcal B_\Psi$ and 
\[
\sum_{n \geq 1} \Vert  \tilde{Q}_n\Vert_{\mathcal{B}_\Psi}<+\infty.
\]
 Furthermore, as in Proposition \ref{QZP}, for any $0<r\leq 1$, the operator  
 $\tilde Q(1) := \sum_{n \geq 1}  \tilde{Q}_n$ is   compact on $\mathcal B_\Psi$  
 and also submarkovian  since $\lambda >0$ and $\lambda'<0$. Its spectral radius
  $\tilde \rho_\Psi $ on  $\mathcal B_\Psi$  is thus strictly les than $1$,
   it is the unique eigenvalue with modulus $\tilde \rho_\Psi$ and  the 
    corresponding eigenspace equals $\mathbb C \tilde H $,  where $\tilde H $ 
    is a positive function in $\mathcal B_\Psi$.

We thus consider the operators  $ ^{\tilde H}{\tilde Q}$ and $\ ^{\tilde H}
{\tilde Q} _n$ for $ n \geq 1$   defined formally by
 \begin{equation*} \label{widetildeQHr}
 \ ^{\tilde H}{\tilde Q} \varphi = \frac{1}{  \tilde \rho_\Psi    \tilde H} 
  {\tilde Q}
( \tilde H\varphi)
  \quad {\rm and} \quad 
  \  {^{\tilde H}} \tilde{Q}_n\varphi = \frac{1}{  \tilde \rho_\Psi   \tilde H} 
   \tilde{Q}_n(  \tilde H\varphi).
 \end{equation*}
 By a direct computation,  as in   Proposition \ref{QZP}, the operators 
  $\  ^{\tilde H}{\tilde Q}_n$ act  on $\mathcal B_{\Psi}$ and 
   the sequence $(n^{3/2} \ ^{\tilde H}{\tilde Q}_n)_{n \geq 0}$ converges in
    $\mathcal  L( \mathcal B_{\Psi})$  to some operator $^{\tilde H} \widetilde 
    {\mathcal E}$; in particular
 for any $x, y \in \mathbb Z$,    
 $$  \ {^{\tilde H}} \tilde{Q}_n(x, y)   \preceq         \frac{{\bf C}_{x, y}}{n^{3/2}}
\quad {\rm and}\quad 
  \ {^{\tilde H}} \tilde{Q}_n(x, y) \sim   \frac{C_{x, y}}{n^{3/2}} 
$$
where $y\mapsto {\bf C}_{x, y} $ and $y\mapsto C_{x, y}$ are positive functions 
in $\mathcal B_\Psi$ for any $x \in \mathbb Z$.

As in  \eqref{iteratesHQ}, for any $\ell \geq 1$  and $\varphi \in \mathcal B_{\Psi}$,
  it holds   that 
\begin{align*}
 {^{\tilde H}} \tilde Q_n ^{(\ell)}(\varphi)&:= \sum_{j_1+ \ldots + j_\ell=n}
  {^{\tilde H}}\tilde Q_{j_1}\ldots {^{\tilde H}} \tilde Q_{j_\ell}(\varphi)
 = 
\frac{1}{ \tilde \rho^\ell_\Psi \tilde H} \tilde Q^{(\ell)}_n( \tilde H \varphi)
\end{align*}
which readily implies
\begin{equation}\label{conjugatetilde}
  \tilde{Q}_n^{(\ell)}\varphi(x)= \tilde \rho^\ell_\Psi \rho^n e^{\lambda x} 
   \tilde H(x) 
{^{\tilde H}} \tilde Q_n ^{(\ell)}(\varphi e^{-\lambda \cdot}/ \tilde H)(x). 
\end{equation}

$\bullet$ \underline{End of the proof}

We use again the decomposition of trajectories \eqref{decompositiontraject}: 
for $x, y\in \mathbb Z $ and $n \geq 1$, 
$$
\mathbb P_x[X_n=y]
  =  \sum_{\ell=0}^{+\infty} \sum _{k=0}^n  Q_k^{(\ell)} (  {\bf V}_{n-k, y})(x),
$$
 where $    {\bf V}_{n-k, y}$ is defined in \eqref{hny} and \eqref{h'ny}. 

 In the  case ({\bf N, P}), the sequences $(n^{3/2} 
   {\bf V}_{n, y}/\rho^n)_{n \geq 1}$ for $y \leq 0 $ and 
 $(n^{3/2}   {\bf V}_{n, y}/{\rho'}^n)_{n \geq 1}$ for $y \geq 1$ 
  converge in $\mathcal B_\Psi$ to some non negative limit function which can
   be explicit as in \eqref{hy} and \eqref{h'y} by  using 
    \eqref{valeurrelativisation}.
    In particular, when $y \geq 1$,  
     the sequences $(n^{3/2}   {\bf V}_{n, y}/{\rho}^n)_{n \geq 1} $
       converge in $\mathcal B_\Psi$ to $0$ when $\rho>\rho'$.  
 By using  \eqref{conjugatetilde}, the expression \eqref{returnTransientcase}  
  of $\mathbb P_x[X_n=y]
  $ in the case  $(\bf Z, P)$  may be rewritten here as 
\begin{equation}\label{returnCaseNP}
\mathbb P_x[X_n=y]
     =
  e^{\lambda x}   \tilde H(x) \rho^n \sum_{\ell=0}^{+\infty} \tilde \rho_\Psi^\ell
 \bigg(\sum _{k=0}^{n}   {^{\tilde H}} \tilde Q  _k^{(\ell)} 
 \bigg(\dfrac{  {\bf V}_{n-k, y}\ e^{-\lambda \cdot}}{\tilde H  \rho^{n-k}}
 \bigg)\bigg)(x).
\end{equation}
 We conclude as in the previous section by applying  Theorem \ref{theogouezel+}, 
  with $\mathcal R_n={^{\tilde H} }\tilde Q_n$. We  obtain $\mathbb P_x[X_n=y]
 \sim C_{x, y} \rho^n/n^{3/2}$ for some strictly positive constant $C_{x, y}$.

 \subsection{The transient case   \texorpdfstring{$\bf  (P, P)$}{PP}}

 The asymptotic behavior   strongly depends on the relative positions of
  the graphs of the Laplace transforms $L$ and $L'$; thus, there are several 
  subcases to consider. Let us first present the general strategy.
 
 Recall that   $L$ and $L'$ reach respectively their  minima at  $\lambda$ and
  $\lambda'$ and that $\rho = L(\lambda)$ and $\rho'=L'(\lambda')$. 
  Since $\mathbb E[\xi_1]$ and  
 $\mathbb E[\xi'_1]$  are both positive, the reals $\lambda$ and $\lambda'$ 
 are both negative.

 As in the case $({\bf N, P)}$, we consider  the   Banach space  $\mathcal B_\Psi$   
  with  
$$
  \ \Psi(x)= \left\{\begin{array}{lll}
e^{ (\vert  \lambda'\vert+\delta) \vert x\vert} & \text{if} & x \leq 0,\\
e^{ (\vert  \lambda\vert+\delta)  x}& \text{if} & x \geq 1.
\end{array}\right.$$
In the sequel, for various values of the real $t$, we  consider the random walks
 $S^{t}$ and ${S'}^{t}$ associated respectively to the  increment distributions 
 $\mu^{t}(\cdot)= e^{ t\cdot}\mu(\cdot)/L(t)$  and ${\mu'}^{ t}(\cdot)= e^{ t\cdot}
 \mu'(\cdot)/L'(t)$; 
these random increments  have respective mean  $\frac{1}{L(t)} \frac{dL}{dt}(t)$ 
and $\frac{1}{L'(t)}\frac{dL'}{dt}(t)$. 

We denote  by $\mathcal X^t=\mathcal X( \mu^{t},  {\mu'}^{t})=   (X_n^t)_{n \geq 0}$ 
the corresponding oscillating random walk  and  $Q^t$ the transition probability of 
 its switching subprocess.  Writing $Q= \sum_{n \geq 1} Q_n$ and  $Q^t= 
 \sum_{n \geq 1} Q^t_n$ as previously,  it holds that, for any $x, y \in \mathbb Z$,
\begin{align*} 
Q_n(x, y)&=
\left\{
\begin{array}{ccccc}
e^{t(x-y)} L(t)^n  Q_n^t(x, y) & {\rm if} & x\leq 0&{\rm and} &y\geq 1,\\
e^{t(x-y)} L'(t)^n Q_n^t(x, y) & {\rm if} & x\geq  1& {\rm and} & y\leq 0.\\
\end{array}
\right.  
\end{align*}
The parameter $t$ is fixed according to the studied subcase; 
its choice   is done   in order to transfer the problem  to  one of the cases 
${\bf (P, N), (Z, Z), (P, Z)},  {\bf (Z, P) }$ or  ${\bf (N, Z) }$ which have
 been previously studied (the case  ${\bf (N, Z) }$ is similar to ${\bf (Z, P)}$, 
 by exchanging the role of $\mu$ and $\mu'$). Once the parameter $t$ is fixed, 
 in order to ``catch" the exponential factor $\max(L(t), L'(t))^n$, we introduce
  and study  another sequence of operators $(\tilde Q_n)_{n \geq 1}$, depending 
  on $t$ and  given by 
 
$\bullet$ if $L(t) \geq L'(t)$,
\begin{align*} 
\tilde Q_n(x, y)&:= 
\left\{
\begin{array}{ccccc}
 Q_n^t(x, y) & {\rm if} & x\leq 0&{\rm and} &y\geq 1,\\
 \left({L'(t)\over L(t)}\right)^n Q_n^t(x, y) & {\rm if} & x\geq  1& {\rm and} & 
 y\leq 0.\\
\end{array}
\right.  
\end{align*}
 
$\bullet$   otherwise, $L(t) < L'(t)$ and
\begin{align*} 
\tilde Q_n(x, y)&:= 
\left\{
\begin{array}{ccccc}
  \left({L(t)\over L'(t)}\right)^n   Q_n^t(x, y) & {\rm if} & x\leq 0&{\rm and} 
  &y\geq 1,\\
 Q_n^t(x, y) & {\rm if} & x\geq  1& {\rm and} & y\leq 0.\\
\end{array}
\right.  
\end{align*}
We set $\displaystyle \tilde Q= \sum_{n \geq 1} \tilde Q_n$. When $L(t)=L'(t)$,  
it holds that $\tilde Q= Q^t$; furthermore,   depending on the values of the drifts of  $\mu^t$ 
and ${\mu'}^t$, this operator  can be either  markovian or submarkovian. 
When $L(t)\neq L'(t)$, the operator $\tilde Q$ is submarkovian.  
When $\tilde Q$  is markovian, we  apply the strategy developed in the recurrent 
cases  $(\bf P, N),  (\bf Z, Z)$ an $(\bf P, Z)$ while in the  submarkovian situations, 
  the spectral radius  of $\tilde Q$ is strictly less than $1$ in $\mathcal B_\Psi$ and
we apply the  strategy developed in the case $(\bf Z, P)$.
\subsubsection{The subcase {\bf A} : $\lambda = \lambda'$} 
 \underline{We fix here $t= \lambda=\lambda'$.}
 
\noindent \underline{\bf A1 : $\rho=\rho'$}
%

In this case, for any $x, y \in \mathbb Z$, it holds that  
\ $
\mathbb P_x[X_n=y]= \rho^n e^{\lambda  (x-y)}\mathbb P_x[X_n^{\lambda}=y], 
$
where $(X_n^{\lambda})_{n \ge 0}$ is a ({\bf Z, Z}) oscillating random walk. Hence 
$\mathbb P_x[X_n^{\lambda}=y]\sim  C_{y} /\sqrt{n}$ by Section \ref{sectionrecurrent}.

  \noindent    \underline{{\bf A2 :}  $\rho>\rho'$}
%
%

   We consider  the sequence of operators $(\tilde{Q}_n)_{n \geq 1}$  
     defined by 
\begin{equation*} 
\tilde{Q}_n(x, y):= 
\left\{
\begin{array}{ccccc}
 \mathbb P[\tau^{\overset{\circ}{S_{\ }} }(x)=n,  
 x+{\overset{\circ}{S_n}}=y ]&
  {\rm if} & x\leq 0&{\rm and} &y\geq 1,\\
   r^n \mathbb P[\tau^{ {\overset{\circ}{S'}}} (x)=n, 
    x+{\overset{\circ}{S'_n}}=y ]
   & {\rm if} & x\geq  1& {\rm and} & y\leq 0.\\
\end{array}
\right.  
\end{equation*}
where $r = \rho'/\rho<1$. We set $\tilde Q= 
\sum_{n \geq1} \tilde Q_n$ and 
notice that $
Q_n^{(\ell)}\varphi(x)= \rho^n e^{\lambda x} 
\tilde{Q}_n^{(\ell)}(\varphi /e^{\lambda \cdot})(x)
$
for any  function $\varphi\in \mathcal B_{\Psi}$ 
and $\ell \geq 1$.
The condition $r<1$ implies that  the operator $\tilde Q$
  is submarkovian  and that its spectral radius $\rho_\Psi$ in
   $\mathcal B_\Psi$ is strictly less than 1.
   By applying  the same scheme as in the 
  transient case $\bf(N, P)$, we  
 obtain   $\mathbb P_x[X_n=y]\sim C_{x,y}\rho^n/n^{3/2}$ 
 as $n \to +\infty$, for some  strictly positive constant 
  $C_{x,y}$. 

  \noindent    \underline {\bf A2   :   $\rho <\rho'$}
 
 Analogously, we set $r=\rho/\rho'<1$ and  define the sequence of 
 operators $(\tilde{Q}_n)_{n \geq 1}$ by 
\begin{equation*} 
\tilde{Q}_n(x, y):=
\left\{
\begin{array}{ccccc}
 r^n \mathbb P[\tau^{ {\overset{\circ}{S_{\ }} }}(x)=n, 
 x+{\overset{\circ}{S_n}}=y]& {\rm if} & x\leq 0&{\rm and} &
   y\geq 1,\\ 
    \mathbb P[\tau^{\overset{\circ}{S'}}(x)=n,
     x+{\overset{\circ}{S'_n} =y}]& 
    {\rm if} & x\geq  1&{\rm and} & y\leq 0.\\
\end{array}
\right.  
\end{equation*}
The rest of the argument is the same as the one in the case 
{\bf A2 :}  $\rho>\rho'$

   \subsubsection{The subcase {\bf B} : $\lambda < \lambda'$}
 
 \noindent    \underline {\bf B1 : $\rho = \rho'$}
 
In this case,  there exists a  unique $\lambda_\star \in
 (\lambda, \lambda')$ 
such that  $L(\lambda_\star)= L'(\lambda_\star)$; 
furthermore $\rho_\star:=
L(\lambda_\star)= L'(\lambda_\star) \in (\rho, 1)$,  
 ${dL\over dt}(\lambda_\star)>0$ 
and ${dL'\over dt}(\lambda_\star)<0$. 

\underline{We fix $t = \lambda_\star$.}
The random walks $S^{\lambda_\star}$ and ${S'}^{\lambda_\star}$ 
 are non centered  
 with  respective means $\frac{1}{\rho_\star}{dL\over dt}
 (\lambda_\star)>0$ and 
 $\frac{1}{\rho_\star}{dL'\over dt}(\lambda_\star)<0$.  
 As a consequence, 
 the oscillating random walk $\mathcal X( \mu^{\lambda_\star},  
 {\mu'}^{\lambda_\star}) $ is ${\bf (P, N)}$; thus, it 
 is positive recurrent 
 on $\mathbb Z$ with the unique invariant probability measure 
 $\nu_\star$. 
Then,  for any $x, y \in \mathbb Z$, and as $n\to +\infty$,
 we  conclude that 
\[
\mathbb P_x[X_n=y]= \rho^n_\star  
e^{\lambda_\star (x-y)}\mathbb P_x[X_n^{\lambda_\star}=y] \sim \rho^n_\star e^{\lambda_\star (x-y)} \nu_\star(y).
\]
 \underline {\bf B2 : $\rho < \rho' \ \&\  L(\lambda')<\rho'$}
 
\underline{We fix $t= \lambda'$} 
and  introduce 
the sequence of operators $(\tilde{Q}_n)_{n \geq 1}$   defined by 
\begin{equation*} 
\tilde{Q}_n(x, y):=
\left\{
\begin{array}{ccccc}
 r^n  \mathbb P[\tau^{S^{\lambda'}}(x)=n,  x+S^{\lambda'}_n=y ]&
  {\rm if} & x\leq 0&{\rm and} & y\geq 1,
 \\
     \mathbb P[\tau^{\overset{\circ}{S'}}(x)=n,   x+{\overset{\circ}{S'_n} =y}]&
      {\rm if} & x\geq 1&{\rm and} &  y\leq 0.\
\end{array}
\right.  
\end{equation*}
with $r = L(\lambda')/\rho' <1$.
The inequality  $  r<1$ implies that 
  the operator $\tilde Q= \sum_{n \geq1} \tilde Q_n$ is 
  submarkovian. The rest of the proof works as in  the 
   case  ${\bf (N, P)}$, via a Doob's transform. This yields the asymptotic 
   behavior $\mathbb P_x[X_n=y]
   \sim C_{x, y} {\rho'}^n/n^{3/2}$  as $n \to +\infty$, 
   where $C_{x, y}>0$. 
   
   \noindent    \underline {\bf B3 : $\rho < \rho' \ \&\  L(\lambda')=\rho'$}
%

\underline{We fix $t= \lambda'$.} 
For any $x, y \in \mathbb Z$, it holds that  
$
\mathbb P_x[X_n=y]= {\rho'}^n e^{\lambda'  (x-y)}\mathbb P_x[X^{\lambda'}_n=y] 
$
where $(X_n^{\lambda'})_{n \ge 0}$ is a {\bf(P, Z)} oscillating 
random walk since
${dL\over dt}(\lambda')>0$.  By Section \ref{sectionrecurrent}, 
it holds that   
$\mathbb P_x[X_n^{\lambda'}=y]\sim  C_{y} /\sqrt{n}$  with $C_y>0$. 
This yields the expected asymptotic
 behavior $\mathbb P_x[X_n=y] \sim C_{x, y} {\rho'}^n/\sqrt{n}$,   with $C_{x,y}>0$, 
 as $n \to +\infty$.

 \noindent    \underline {\bf B4 : $\rho < \rho' \ \& \ L(\lambda')>\rho'$ }

In this case,  there exists a  unique $\lambda_\star \in (\lambda, \lambda')$ such 
that  $L(\lambda_\star)= L'(\lambda_\star)=: \rho_\star$. Furthermore  
$\rho_\star \in (\rho, 1)$,  ${dL\over dt}(\lambda_\star)>0$ and
 ${dL'\over dt}(\lambda_\star)<0$. \underline{We fix $t = \lambda_\star$} 
 and adopt the same argument as in  the case {\bf B1}. 
  
 \noindent    \underline {\bf B5 : $\rho > \rho' \ \& \ L'(\lambda)<\rho$ }
 
\underline{We fix $t= \lambda$} and introduce the sequence of 
operators $(\tilde{Q}_n)_{n \geq 1}$  where $\tilde{Q}_n$ is defined by 
\begin{equation*} 
\tilde{Q}_n(x, y):=
\left\{
\begin{array}{ccccc}
 \mathbb P[\tau^{  \overset{\circ}{S_{\ }} }(x)=n,  x+{\overset{\circ}{S_{\ }} }_n=y 
 ]& {\rm if} & x\leq 0&{\rm and} & y\geq 1,\\
    r^n  \mathbb P[\tau^{ {S'}^{\lambda}}(x)=n,  x+{S'_n}^{\lambda}=y ]& {\rm if} & 
     x\geq 1 &{\rm and} & y\leq 0,\\
\end{array}
\right.  
\end{equation*}
with $r = L'(\lambda)/\rho <1$.
The inequality   $r<1$ implies that  the operator
  $\tilde Q= \sum_{n \geq1} \tilde Q_n$ is  submarkovian  
  on $\mathcal B_\Psi$.
    The  rest of the proof works as in  the 
    case   ${\bf (N, P)}$, via a Doob's transform. 
    This yields the asymptotic behavior $\mathbb P_x[X_n=y] 
    \sim C_{x, y} 
    \rho ^n/n^{3/2}$ 
   as $n \to +\infty$, with $C_{x, y}>0$.

 \noindent    \underline {\bf B6 : $\rho > \rho' \ \& \ L'(\lambda)=\rho$ }

\underline{We fix $t= \lambda$.}
For any $x, y \in \mathbb Z$, it holds that  
$
\mathbb P_x[X_n=y]= {\rho}^n e^{\lambda   (x-y)}
\mathbb P_x[X^{\lambda}_n=y] 
$
where $(X_n^{\lambda})_{n \ge 0}$ is a ${\bf(Z, N)}$ oscillating 
random walk since
${dL'\over dt}(\lambda)<0$.  By Section \ref{sectionrecurrent}
 (notice that the case $({\bf Z, N)}$ is analogous to $({\bf P,  Z)} $
  by exchanging the role of $\mu$ and $\mu'$), it holds that   
$\mathbb P_x[X_n^{\lambda}=y]\sim  C_{y}/\sqrt{n} $.
This yields the asymptotic behavior $\mathbb P_x[X_n=y]\sim C_{x, y} 
{\rho}^n/\sqrt{n}$ as $n \to +\infty$, with $C_{x, y} >0$..

 \noindent    \underline {\bf B7 : $ \rho > \rho'  \ \& \ L'(\lambda)>\rho$ }

In this case,  there exists a  unique $\lambda_\star \in (\lambda, \lambda')$ such that  $L(\lambda_\star)= L'(\lambda_\star)=: \rho_\star$;  moreover $\rho_\star \in (\rho, 1)$,  ${dL\over dt}(\lambda_\star)>0$ and ${dL'\over dt}(\lambda_\star)<0$. \underline{We fix $t = \lambda_\star$} and we apply the same argument as in  the cases   {\bf B1} and {\bf B4}.

   \subsubsection{The subcase {\bf C} : $\lambda > \lambda'$}
   
  As in  the  case {\bf B}, we have  totally  7 subcases to consider: 
\begin{center}
   \begin{tabular}{|c|c|}
   \hline
\ & {\small Condition}   
 \\
\hline \hline 
{\bf C1}  & $\rho=\rho'$   
 \\
  \hline \hline
  {\bf C2} &$\rho<\rho'  \  \& \ L(\lambda')< \rho'$     
            \\
      \hline
  {\bf C3} &$\rho<\rho'  \  \& \ L(\lambda')=\rho'$    
  \\       
            \hline
  {\bf C4} &$\rho<\rho'  \  \& \ L(\lambda')> \rho'$     
            \\
   \hline \hline
       {\bf C5} &$\rho>\rho'  \  \& \ L'(\lambda)< \rho$    
            \\
      \hline
  {\bf C6} &$\rho>\rho'  \  \& \ L'(\lambda)=\rho$   
  \\       
            \hline
  {\bf C7} &$\rho>\rho'  \  \& \ L'(\lambda)> \rho$    
            \\
   \hline
   \end{tabular}
  \end{center}
 
 $\bullet$ The proof in the subcases {\bf C2} and {\bf C5} is
  essentially  the same  as in {\bf B2} or {\bf B5} with
   \underline{ $t=\lambda'$}.
 
 $\bullet$ In the subcases {\bf C1}, {\bf C4} and {\bf C7}, we  choose 
 \underline{$t=\lambda_\star$} as for the subcases 
  {\bf B1}, {\bf B4} and {\bf B7}. Hence
 \[
\mathbb P_x[X_n=y]= \rho^n_\star  e^{\lambda_\star 
(x-y)}\mathbb P_x[X_n^{\lambda_\star}=y]
\]
The condition $\lambda>\lambda'$ implies  that 
 ${dL\over dt}(\lambda_\star)<0$ and ${dL'\over dt}
 (\lambda_\star)>0$. Hence, the  oscillating random walk 
 $\mathcal X( \mu^{\lambda_\star},  {\mu'}^{\lambda_\star}) $
  is ${\bf  (N, P)}$ and falls within the transient framework
   studied in Subsection \ref{subsectionNP}. 
 Notice that  the Laplace transform $L_{\mu^{\lambda_\star}}$  
  of the measure $\mu^\star$   reaches its minimum at $\lambda-\lambda_\star$  and equals $\rho/\rho_\star$ (with a similar statement for ${\mu'}^\star$).  
Consequently,   as $n\to +\infty$, 
\[
\mathbb P_x[X_n^{\lambda_\star}=y]\sim   \left({\max(\rho, \rho')\over \rho_\star}\right)^n/n^{3/2} \quad (\text{up to a multiplicative positive constant})
\]

  $\bullet$ In the subcase {\bf C3},  we  choose 
   \underline{$t=\lambda' $}. 
 The condition  $\lambda>\lambda'$ implies that 
  ${dL\over dt}(\lambda')<0$ and ${dL'\over dt}(\lambda')=0$. Hence,
   the  oscillating random walk  $\mathcal X( \mu^{\lambda'},
    {\mu'}^{\lambda'})$  is ${\bf  (N, Z)}$   (which is analogous to 
    the case ${\bf  (Z, P)}$) and  falls within the transient
     framework studied in Subsection \ref{sectiontransientZP}. Hence
 for any $x, y \in \mathbb Z$, it holds that  
$
\mathbb P_x[X_n=y]= {\rho'}^n  
 e^{\lambda'(x-y)}\mathbb P_x[X_n^{\lambda'}=y]
$ 
with $\mathbb P_x[X_n^{\lambda'}=y]\sim C_{x, y}/n^{3/2}$ as 
$n \to +\infty$, with $C_{x, y} >0$.

$\bullet$ In the subcase  {\bf C6}, we choose 
\underline{$t=\lambda$}. Since $\lambda>\lambda'$, it 
holds that  ${dL\over dt}(\lambda)=0$ and ${dL'\over dt}(\lambda)>0$. Hence,
 the  the oscillating random walk $\mathcal X( \mu^{\lambda}, 
 {\mu'}^{\lambda}) $ is ${\bf  (Z, P)}$ and falls within the transient
  framework studied in subsection \ref{sectiontransientZP}. Hence
 for any $x, y \in \mathbb Z$, it holds that  
$
\mathbb P_x[X_n=y]= \rho^n  e^{\lambda (x-y)}\mathbb P_x[X_n^{\lambda}=y]
$ 
with $\mathbb P_x[X_n^{\lambda}=y]\sim C_{x, y}/n^{3/2}$ as $n \to +\infty$, with $C_{x, y}>0$.



\end{document}